\documentclass[11pt,oneside,reqno]{amsart}

\usepackage{geometry}
\usepackage{color}
\usepackage{amssymb}
\usepackage{amsmath}
\usepackage{arydshln}
\usepackage{hyperref}
\usepackage{bbm}
\usepackage{mathrsfs}

\def\BibTeX{{\rm B\kern-.05em{\sc i\kern-.025em b}\kern-.08em
    T\kern-.1667em\lower.7ex\hbox{E}\kern-.125emX}}

\numberwithin{equation}{section}
\numberwithin{figure}{section}

\newtheorem{satz}{Satz}[section]
\newtheorem{Theorem}[satz]{Theorem}
\newtheorem{Proposition}[satz]{Proposition}

\newtheorem{Corollary}[satz]{Corollary}
\newtheorem{lem}[satz]{Lemma}
\newtheorem{Lemma}[satz]{Lemma}

\newtheorem{Remark}[satz]{Remark}

\newtheorem{Example}[satz]{Example}

\renewcommand{\Re}{\operatorname{Re}}
\renewcommand{\Im}{\operatorname{Im}}
\newcommand{\R}{\mathbb{R}}
\newcommand{\E}{\mathbb{E}}
\renewcommand{\P}{\mathbb{P}}
\newcommand{\e}{\varepsilon}

\begin{document}

\title[On the anisotropic stable JCIR process]{On the anisotropic stable JCIR process}

\author{Martin Friesen}
\address[Martin Friesen]{Faculty for Mathematics and Natural Sciences, University of Wuppertal, Germany}
\email{friesen@math.uni-wuppertal.de}

\author[Peng Jin]{Peng Jin\textsuperscript{*}}
\thanks{\textsuperscript{*}Peng Jin is supported by the STU Scientific Research Foundation for Talents (No. NTF18023).}
\address[Peng Jin]{Department of Mathematics, Shantou University, Shantou, Guangdong 515063, China}
\email{pjin@stu.edu.cn}

\date{\today}



\begin{abstract}
 We investigate the anisotropic stable JCIR process
 which is a multi-dimensional extension of the stable JCIR process
 but also a multi-dimensional analogue of the classical JCIR process.
 We prove that the heat kernel of the anisotropic stable JCIR process exists and it satisfies an a-priori bound
 in a weighted anisotropic Besov norm. Based on this regularity result we deduce the strong Feller property and prove, for the subcritical case, exponential ergodicity in total variation. Also, we show that in the one-dimensional case the corresponding heat kernel is smooth.
\end{abstract}

\maketitle

\allowdisplaybreaks

\section{Introduction}
The classical JCIR process is a commonly used building block for different models in mathematical finance, see \cite{A15}.
For given $b, \sigma \geq 0$ and $\beta \in \R$ it is
obtained as the unique $\R_+$-valued strong solution to
\[
 \mathrm{d}X^x(t) = (b + \beta X(t))\mathrm{d}t + \sqrt{\sigma X(t)} \mathrm{d}B(t) + \mathrm{d}J(t), \qquad X^x(0) = x \geq 0,
\]
where $(B(t))_{t \geq 0}$ is a one-dimensional Brownian motion and
$(J(t))_{t \geq 0}$ is a L\'evy subordinator on $\R_+$ that is independent of the Brownian motion. For a particular choice of subordinator $(J(t))_{t \geq 0}$ such a process was first introduced in \cite{DG01}. Some of its specific properties were  studied in \cite{JKR17b}, see also the references therein.
Replacing the Brownian motion $(B(t))_{t \geq 0}$
by a spectrally positive $\alpha$-stable L\'evy process $(Z^{\alpha}(t))_{t \geq 0}$ whose symbol is given, for $\alpha \in (1,2)$, by
\begin{align}\label{alpha stable symbol}
\Psi_{\alpha}(\xi) = \int_{0}^{\infty}\left( \mathrm{e}^{\mathrm{i}\xi z} - 1 - \mathrm{i}\xi z \right)\mu_{\alpha}(\mathrm{d}z),
\qquad  \mu_{\alpha}(\mathrm{d}z) = \mathbbm{1}_{\R_+}(z)\frac{1}{c(\alpha)} \frac{\mathrm{d}z}{z^{1 + \alpha}},
\end{align}
and replacing the square-root by $\sqrt[\alpha]{\cdot}$
one obtains the stable JCIR process
\begin{align}\label{stable CIR}
 \mathrm{d}X^x(t) = (b + \beta X^x(t))\mathrm{d}t +  \sqrt[\alpha]{\sigma X^x(t)} \mathrm{d}Z^{\alpha}(t)
 + \mathrm{d}J(t), \ \ X^x(0) = x \geq 0.
\end{align}
Note that the normalization constant
$c(\alpha) = \int_0^{\infty}\left(e^{ - z} -1 + z\right)z^{-1- \alpha}\mathrm{d}z$
is chosen in such a way that $\Psi_{\alpha}(\mathrm{i}\xi) = \xi^{\alpha}$ for $\xi \geq 0$. This process is a special case of the short-rate models used in  \cite{JMS17, JMSZ18, CLP18a}. One important advantage of these models is their analytical tractability as many desired expressions (e.g. the Laplace transform) can be computed explicitly.

In this work we study the \textit{anisotropic stable JCIR process},
i.e., the multi-dimensional analogue of the stable JCIR process \eqref{stable CIR}, obtained as the unique $\mathbb{R}_{+}^{m}$-valued strong solution to the system of stochastic equations
\begin{equation}\label{eq: a roots model}
    \mathrm{d}X^x_k(t) = \left( b_k + \sum_{j=1}^{m}\beta_{kj}X^x_j(t)\right)\mathrm{d}t + \sqrt[\alpha_k]{\sigma_k X^x_k(t)}\mathrm{d}Z_k(t) + \mathrm{d}J_k(t),
\end{equation}
where $k \in \{1,\dots, m\}$, $X^x(0) = x \in \R_+^m$,
$b = (b_1,\dots, b_m), (\sigma_1,\dots, \sigma_m) \in \R_+^m$
and $\beta = (\beta_{jk})_{j,k \in \{1,\dots, m\}}$
is such that $\beta_{jk} \geq 0$ for all $j \neq k$.
Here $Z_{1},\ldots,Z_{m}$ are independent and each $Z_k$, $k=1,\ldots,m$, is a one-dimensional spectrally positive $\alpha_k$-stable
L\'evy process with symbol $\Psi_{\alpha_k}$ as in \eqref{alpha stable symbol},
where $\alpha_1,\dots, \alpha_m \in (1,2)$.
The process $J$, which is independent of $Z = (Z_1,\dots, Z_m)$, is a L\'evy subordinator on $\R_+^m$, i.e., its  L\'evy measure $\nu$ is supported on $\R_+^m$ and $J$ has symbol
\[
 \Psi_J(\xi) = \int_{\R_+^m}\left( \mathrm{e}^{\mathrm{i}\langle \xi, z \rangle} - 1 \right)\nu(\mathrm{d}z), \qquad \int_{\R_+^m}\min\{1,|z|\} \nu(\mathrm{d}z) < \infty.
\]

It follows from \cite{BLP15} that \eqref{eq: a roots model} has a unique $\R_+^m$-valued strong solution. Moreover,
this process is an affine process on state space $\mathbb{R}_{+}^{m}$
(see \cite{DFS03, BLP15}) whose characteristic function satisfies
 \begin{align}\label{affine property}
  \E[ \mathrm{e}^{\langle u, X^x(t) \rangle } ]
   = \mathrm{e}^{ \phi(t, u) + \langle x, \psi(t, u)\rangle}, \qquad x \in \R_+^m,
 \end{align}
 where $u \in \mathbb{C}^m$ is such that $\mathrm{Re}(u) \leq 0$.
 Here $\phi$ and $\psi = (\psi_1,\dots, \psi_m)$ are the unique solutions to the generalized Riccati equations
 \begin{align} \label{eq:riccati}
  \begin{cases} \partial_t \phi(t,u) = F(\psi(t,u)), & \phi(t,0) = 0,
     \\ \partial_t \psi(t,u) = R(\psi(t,u)), & \psi(t,0) = u,
 \end{cases}
 \end{align}
 where $F$ and $R = (R_1, \dots, R_m)$ are given by
 \begin{align*}
     F(u) &= \langle b, u \rangle + \int_{\R_+^m} \left(\mathrm{e}^{\langle u, z\rangle} - 1\right) \nu(\mathrm{d}z),
     \\   R_j(u) &= \sum_{k=1}^{m}\beta_{kj}u_k
     + \int_{0}^{\infty}\left( \mathrm{e}^{u_j z} - 1 - u_jz \right) \mu_{\alpha_j}(\mathrm{d}z).
 \end{align*}
Following the general theory of affine processes it
can be shown that $(X^x(t))_{t \geq 0}$ is a Feller process,
that its transition semigroup
acts on the Banach space of continuous functions vanishing at infinity,
and that its generator $(L, D(L))$ has core $C_c^{\infty}(\R_+^m)$ and for $f \in C_c^{\infty}(\R_+^m)$
\begin{align*}
 Lf(x) &= \langle b + \beta x, \nabla f(x) \rangle
 + \int_{\R_+^m}\left( f(x + z) - f(x) \right) \nu(\mathrm{d}z)
 \\ &\ \ \ + \sum_{j=1}^{m} \sigma_j x_j \int_{0}^{\infty}\left( f(x + e_jz) - f(x) - z \frac{\partial f(x)}{\partial x_j} \right) \mu_{\alpha_j}(\mathrm{d}z),
\end{align*}
where $e_1,\ldots,e_m$ denote the canonical basis vectors in $\R^m$.

The purpose of this work is twofold.
Firstly, we investigate regularity of the heat kernel including a very simple proof of the strong Feller property, and secondly,
based on the obtained results we study the convergence to equilibrium in total variation.
On the way proving these results we also obtain non-extinction for the anisotropic stable JCIR process in the spirit of \cite{FU13, DFM14, FJR19b}.

One commonly used method to study existence and smoothness of heat kernels is based on Malliavin calculus, see e.g. \cite{BC86, D11, P96} and the references therein.
Concerning other analytical methods we refer to \cite{BSK17, C18, KK18, KR17, KR19, KRS18} where some interesting progress for stochastic equations driven by cylindrical L\'evy processes has been obtained.
Having in mind that the anisotropic stable JCIR process has no diffusion component, that the L\'evy measure of the driving noise is singular and has no second moments, and finally that the volatility coefficients in \eqref{eq: a roots model} are merely H\"older continuous and degenerate at the boundary,
it is not clear how the aforementioned techniques could be applied in the setting of this paper. Based on the affine structure of the process it is reasonable to study the heat kernel by Fourier methods similarly to \cite{FMS13}, where affine processes with non-degenerate diffusion component were treated, or by spectral expansions in the spirit of \cite{CLP18}. While Fourier methods turn out to be adequate for proving
existence of a smooth density for the one-dimensional stable JCIR process (see Section 3),
it seems difficult to extend them to the anisotropic framework with absent diffusion component. Moreover, it would be interesting
to extend the techniques developed in \cite{CLP18} to this multi-dimensional setting.
In contrast, our approach for the study of the multi-dimensional case
is based on a suitable short-time approximation of the process combined with a discrete integration by parts in the spirit of
\cite{DF13, R17, FJR18, FJR18a}.
Since these methods do not use the affine structure of the process,
they can be applied to other Markov processes as well.

The long-time behavior of one-dimensional affine processes with state space $\R_+$
was studied in \cite{KM12}, \cite[Chapter 3]{L11} and \cite{LM15}.
Results applicable to a class of non-affine Markov processes on $\R_+$
have been recently obtained in \cite{FJKR19}.
The coupling method in \cite{LM15} is very effective for $1$-dimensional continuous-state branching processes with immigration. However, it used the fact the the extinction time of a continuous-state branching process can be estimated via the Laplace transform of the process. It is not clear if this approach can be extended to higher dimensional cases. For subcritical OU-type processes and 1-dimensional
continuous-state branching processes with immigration,
the exponential ergodicity in total variation has been derived under rather general conditions, see \cite{W12,LM15} and \cite{FJKR19},
all of which used coupling techniques.
Other than these two cases, only very few
results on ergodicity in total variation are available for multi-dimensional
affine processes, except for the models treated in \cite{BDLP14,JKR17, MSV18,ZG18}.
The reason is as follows: in the general case it is not clear if the
powerful coupling technique (see \cite{W12, LM15}) still works;
also, it remains a difficult problem to verify the irreducibility of the process when applying the Meyn-Tweedie method (see \cite{MT09}).
To overcome these difficulties we use instead a Harris-type theorem based on a local Dobrushin condition, see Theorem \ref{th:3} and \cite{H16, Kulik18}.
In order to verify the local Dobrushin condition
we use continuity (regularity) of the heat kernel combined with some weak form of irreducibility similarly to \cite{PZ18}.
At this point it is worthwhile to mention that the verification of the local Dobrushin condition does not require the full strength of our regularity result.
Indeed, one could apply \cite[Proposition 2.9.1 and Remark 2.9.2]{Kulik18} for which the Besov regularity from Section 4 is sufficient.
This work seems to provide the first result on ergodicity in total variation
for multi-dimensional affine processes which does not rely on smoothing properties of the diffusion component. Moreover, the method of this paper
can be also applied to non-affine Markov processes.

This paper is organized as follows.
In Section 2 we state and discuss the main results of this work.
Regularity of the heat kernel for the one-dimensional stable JCIR process as in \eqref{stable CIR} is discussed in Section 3.
Regularity of the anisotropic stable JCIR process is studied in Section 4,
while ergodicity in total variation is proved in Section 5.
Finally, some auxiliary results and general theory on ergodicity of Markov processes are collected in the appendix.

\section{Statement of results}

\subsection{Existence and smoothness of the heat kernel in dimension $m = 1$}
Let $(X^x(t))_{t \geq 0}$ be the one-dimensional stable JCIR process, i.e.,
the unique $\R_+$-valued strong solution to \eqref{stable CIR}
and $P_t(x,\mathrm{d}y)$ its transition probability kernel.
The following is our first main result.
\begin{Theorem}\label{th: regularity in one dimension}
Suppose that there exist constants $C, M > 0$ and  $\vartheta\in(\alpha-1,1]$ such that
\begin{equation}
b\xi+\int_{0}^{\infty}\left(1-\mathrm{e}^{-z\xi}\right)\nu(\mathrm{d}z)\geq C\xi^{\vartheta}, \qquad \xi \geq M.\label{eq: condition for immigraion}
\end{equation}
Then for each $t>0$ and each $x\geq0$, the heat
kernel $P_{t}(x,\mathrm{d}y)$ has density $p_{t}(x,y)$ which is jointly continuous in $(t,x,y) \in (0,\infty) \times [0,\infty)^2$. Moreover, for each $t > 0$,
the function $\R_+ \times \R_+ \ni (x,y) \longmapsto p_{t}(x,y)$ is smooth and
\[
 \sup_{(x,y) \in \R_+ \times \R_+}|\partial^n_x \partial^k_y p_t(x,y)| < \infty, \qquad \forall  n,k \in \mathbb{N}_0.
\]
\end{Theorem}
Condition \eqref{eq: condition for immigraion} is natural to guarantee that the process does not hit the boundary and hence $P_t(x,\mathrm{d}y)$ has no atom at the boundary, i.e. $p_t(x,y)$ is also continuous at $y = 0$.
Let us refer to \cite{FU13, DFM14} for some related results.
If $b > 0$, then \eqref{eq: condition for immigraion} is satisfied with $\vartheta = 1$ and $C = b$. In the case $b = 0$ condition \eqref{eq: condition for immigraion} is still satisfied provided that the subordinator $J$ has sufficiently many small jumps (see the examples at the end of this section).

The proof of Theorem \ref{th: regularity in one dimension} is given in Section 3 and deeply relies on the affine structure of the process (see \eqref{affine property}),
i.e. we exploit the fact that its characteristic function satisfies
\begin{align}\label{one dimensional affine}
\E[\mathrm{e}^{uX^{x}(t)}]=\mathrm{e}^{\phi(t,u)+x\psi(t,u)},\qquad t\geq0,\ x\geq0,
\end{align}
where $u\in\mathbb{C}$ is such that $\Re(u)\leq0$, and $\phi,\psi$
solve uniquely the generalized Riccati equations
\begin{align}\label{one dimensional riccati equation}
\begin{cases}
\partial_{t}\phi(t,u)=b\psi(t,u)+\int_{0}^{\infty}\left(e^{z\psi(t,u)}-1\right)\nu(\mathrm{d}z), \ \ \phi(0,u) = 0\\
\partial_{t}\psi(t,u)=\beta\psi(t,u)+\int_{0}^{\infty}\left(\mathrm{e}^{z\psi(t,u)}-1-\psi(t,u)z\right)\mu_{\alpha}(\mathrm{d}z), \ \ \psi(0,u) = u.
\end{cases}
\end{align}
We deduce the assertion by showing enough integrability for the characteristic function
$\R \ni u \longmapsto \E[\mathrm{e}^{\mathrm{i}uX^{x}(t)}]$.
For this purpose we adapt some ideas from the multi-dimensional diffusion case studied in \cite{FMS13}, where a H\"ormander-type condition on the drift and diffusion parameters is imposed.

\subsection{Existence of heat-kernel and strong Feller property in dimension $m \geq 1$}
Here and below we denote by $(X^x(t))_{t \geq 0}$ the anisotropic stable JCIR process obtained from \eqref{eq: a roots model} with initial condition $X^x(0) = x \in \R_+^m$,
and recall that it depends on the parameters $b \in \R_+^m, \sigma_1,\dots, \sigma_m \geq 0$, $\alpha_1,\dots, \alpha_m \in (1,2)$, $(\beta_{kj})_{k,j = 1,\dots, m}$ with $\beta_{kj} \geq 0$ for $k \neq j$, and a L\'evy subordinator $\nu(\mathrm{d}z)$ on $\R_+^m$. Finally, let us assume that $\sigma_1, \dots, \sigma_m > 0$.
The case where $\sigma_i = 0$ holds for some $i \in \{1,\dots, m\}$ can
be also studied by the methods of this paper provided we assume an additional ''non-degeneracy'' condition on the one-dimensional L\'evy process $(J_i(t))_{t \geq 0}$. However, in order to keep the arguments simple and neat we decided to exclude these cases.
The following condition is a multi-dimensional analogue
of \eqref{eq: condition for immigraion}.

\begin{enumerate}
 \item[(A)] There exist constants $C, M > 0$ and $\vartheta_1,\dots, \vartheta_m$ such that for all $k = 1, \dots, m$, $\vartheta_k \in (\alpha_k -1,1]$ and
 \begin{align}\label{cond: A}
   b_k \xi + \int_{\R_+^m}\left( 1 - \mathrm{e}^{- \xi z_k}\right)\nu(\mathrm{d}z) \geq C \xi^{\vartheta_k}, \qquad \forall \xi \geq M.
 \end{align}
\end{enumerate}

\begin{Remark}
 If $b \in \R_{++}^m = \{ x \in \R_+^m \ | \ x_1,\dots, x_m > 0\}$,
 then condition (A) is satisfied. If $b \in \partial \R_+^m$,
 then condition (A) is still satisfied provided that the L\'evy process $J$ has sufficiently many jumps in direction $k$ with $b_k = 0$. Some particular examples satisfying condition (A) with $b \in \partial \R_+^m$ are given in the end of this section.
\end{Remark}
 Condition (A) guarantees that the process has a sufficiently strong drift pointing inwards (i.e. in the interior $\R_{++}^m$) and hence does not hit the boundary of its state space, see Section 4 for additional details.
 The next remark states that (A) imposes essentially a condition that is independent of  the big jumps of the subordinator $J$.
 \begin{Remark}
  Let $b \in \R_+^m$ and let $\nu$ be a L\'evy measure on $\R_+^m$. Then condition (A) is satisfied for $b,\nu$ if and only if it is satisfied for $b,\mathbbm{1}_{ \{ |z| \leq 1\} }\nu(\mathrm{d}z)$.
 \end{Remark}
 The following is our main regularity result for the heat kernel of the anisotropic stable JCIR process.
\begin{Theorem}\label{corr L1 continuity}
 Suppose that condition (A) is satisfied. Then $P_t(x,\mathrm{d}y) = p_t(x,y)\mathrm{d}y$ and
 \[
  \R_+^m \ni x \longmapsto p_t(x,\cdot) \in L^1(\R_+^m)
 \]
 is continuous for each $t > 0$.
 In particular, the anisotropic stable JCIR process has the strong Feller property.
 \end{Theorem}
The proof of this result is given in Section 4
and is divided into 4 steps.
Namely, we first prove existence of a heat kernel under an additional moment condition for $\nu$ and provide an estimate in a suitably weighted anisotropic Besov norm which takes also the behavior of the process at the boundary into account. Secondly, we estimate uniformly the probability that the process hits its boundary in positive time. Then, with the same  moment condition for $\nu$, we deduce the assertion from a compactness argument combined with previous two steps. Finally, we use a convolution trick to remove the extra moment assumption and prove the assertion in the general case. The same approach can also be applied to general affine (and non-affine) processes.

 \subsection{Exponential ergodicity in total variation}
The anisotropic stable JCIR process is called \textit{subcritical}, if
$\beta = (\beta_{jk})_{j,k \in \{1,\dots, m\}}$ has only eigenvalues with negative real-parts.
Assuming that the anisotropic stable JCIR process is subcritical and satisfies
\begin{align}\label{nu log moment}
  \int_{\R_+^m} \mathbbm{1}_{ \{ |z| > 1\} } \log(1 + |z|) \nu(\mathrm{d}z) < \infty,
\end{align}
existence, uniqueness and a representation of the characteristic function for
the invariant measure $\pi$ was first obtained in \cite{JKR18} where stability for
the corresponding Riccati equations was investigated. Then
\begin{align}\label{log moment pi}
 \int_{\R_+^m} \log( 1 + |x| ) \pi(\mathrm{d}x) < \infty
\end{align}
and an exponential rate of convergence
for $P_t(x,\cdot) \longrightarrow \pi$ in different Wasserstein distances
was shown in \cite{FJR19a} where affine processes on the canonical state space have been obtained as unique strong solutions to a system of stochastic equations. For one-dimensional affine processes on $\R_+$ regularity (and other properties)
of the invariant measure $\pi$ was studied in \cite{CLP18, KM12}.
Using the regularity for the heat kernel obtained in Theorem \ref{corr L1 continuity} we prove exponential ergodicity in the total variation norm
\begin{align}\label{eq:01}
 \| \rho \|_{\mathrm{TV}} = \sup \limits_{A \in \mathcal{B}(\R_+^m)}|\rho|(A)
 = \sup \limits_{\| f\|_{\infty} \leq 1} \left| \int_{\R_+^m}f(x) \rho(\mathrm{d}x)\right|,
\end{align}
where $|\rho| = \rho^+ + \rho^-$ and $\rho^{\pm}$ denote the Hahn-Jordan decomposition of a signed Borel measure $\rho$ on $\R_+^m$.
Our last main result provides a sufficient condition for the exponential ergodicity in the stronger total variation distance.
\begin{Theorem}\label{th:main result}
 Suppose that the anisotropic stable JCIR process is subcritical, satisfies condition (A) and \eqref{nu log moment}.
 Then there exist constants $C, \delta > 0$ such that for all $t \geq 0$ and $x \in \R_+^m$
 \[
  \| P_t(x,\cdot) - \pi \|_{\mathrm{TV}} \leq C \left( 1 + \log(1 + |x|) + \int_{\R_+^m}\log(1+|y|)\pi(\mathrm{d}y)\right) \mathrm{e}^{- \delta t}.
 \]
\end{Theorem}
The proof of this theorem is based on a Harris-type theorem and is given in Section 5.
It basically requires to check a local Dobrushin condition and a Foster-Lyapunov drift condition for the extended generator.
The local Dobrushin condition is deduced from the regularity results from Theorem \ref{corr L1 continuity}
combined with a weak form of irreducibility similar to \cite{PZ18}.
Finally, the Foster-Lyapunov condition can be checked by direct computation
combined with a convolution argument similar to \cite{FJR19a, JKR18}.

\subsection{Examples for main conditions}
Recall that condition (A) is satisfied, if $b \in \R_{++}^m$.
So let us consider the case $b \in \partial \R_+^m$.
For simplicity, we suppose that $b = 0$ and provide conditions on $\nu$ such that
(A) is still satisfied. Set
\[
\alpha_{\max}:=\max \{\alpha_1,\dots,\alpha_m\},\quad \alpha_{\min}:=\min \{\alpha_1,\dots,\alpha_m\}.
\]
\begin{Example}
 Let $\nu$ be given by the spherical decomposition
 \begin{align}\label{spherical decomposition nu}
  \nu(A) = \int_0^{\infty} \int_{S_+^{m-1}}\mathbbm{1}_{A}(r\sigma)\lambda(\mathrm{d}\sigma)\frac{\mathrm{d}r}{r^{1 + \vartheta}}
 \end{align}
 where $\vartheta \in (\alpha_{\max}-1,1)$, $S_+^{m-1} = \{ \sigma \in \R_+^m \ | \ |\sigma| = 1\}$, and $\lambda$ is a measure on $S_+^{m-1}$. Then we obtain
 \[
  \int_{\R_+^m} \left( 1 - \mathrm{e}^{- \xi z_k} \right)\nu(\mathrm{d}z)
  = \xi^{\vartheta} \int_{S_+^{m-1}}\sigma_k^{\vartheta} \lambda(\mathrm{d}\sigma) \int_{0}^{\infty}\left(1 - \mathrm{e}^{- r} \right)\frac{\mathrm{d}r}{r^{1+\vartheta}}.
 \]
 Hence (A) holds, if $\int_{S_+^{m-1}}\sigma_k^{\vartheta} \lambda(\mathrm{d}\sigma) > 0$.
 This includes the following cases:
 \begin{enumerate}
     \item[(a)] If $\lambda(\mathrm{d}\sigma) = \mathbbm{1}_{S_+^{m-1}}(\sigma)\mathrm{d}\sigma$ is the uniform distribution on $S_+^{m-1}$, then
     \[
      \nu(\mathrm{d}z) = \mathbbm{1}_{\R_+^m}(z) \frac{\mathrm{d}z}{|z|^{d+\vartheta}}.
     \]
     \item[(b)] If $\lambda(\mathrm{d}\sigma) = \sum_{k=1}^{m}\delta_{e_k}(\mathrm{d}\sigma)$, then
     \[
      \nu(\mathrm{d}z) = \sum_{k=1}^{m} \mathbbm{1}_{\R_+}(z_k)\frac{\mathrm{d}z_k}{z_k^{1 + \vartheta}} \otimes \prod_{j \neq k} \delta_{0}(\mathrm{d}z_j).
     \]
 \end{enumerate}
\end{Example}
 The next example shows that the stability index $\vartheta$ appearing in \eqref{spherical decomposition nu}
 is also allowed to depend on the direction of the jump.
\begin{Example}
 Let $J(t) = (J_1(t), \dots, J_m(t))$ where $J_1,\dots, J_m$ are independent
 L\'evy subordinators on $\R_+$ with L\'evy measures $\mathbbm{1}_{\R_+}(z_k)z_k^{-1-\vartheta_k}\mathrm{d}z_k$ with $\nu_k \in (\alpha_k-1,1)$and $k = 1,\dots, m$. Then $J$ has L\'evy measure
 \[
  \nu(\mathrm{d}z) = \sum_{k=1}^{m} \mathbbm{1}_{\R_+}(z_k) \frac{\mathrm{d}z_k}{z_k^{1+\vartheta_k}} \otimes \prod_{j \neq k}\delta_0(\mathrm{d}z_j)
 \]
 and for $\xi \geq 0$ it holds that
 \[
  \int_{\R_+^m} \left( 1 - \mathrm{e}^{- \xi z_k} \right)\nu(\mathrm{d}z)
  = \xi^{\vartheta_k} \int_{0}^{\infty}\left( 1 - \mathrm{e}^{- r}\right)\frac{\mathrm{d}r}{r^{1+\vartheta_k}}.
 \]
 In particular condition (A) is satisfied.
\end{Example}
We may also easily find examples where in some directions $b_k > 0$ while for other directions $b_k = 0$ and the L\'evy measure $\nu$ has sufficiently many jumps
(e.g. it is given by previous two examples).
Our last example provides a deviation from anisotropic stable L\'evy measures.
\begin{Example}
 Take $\vartheta_k \in (\alpha_k -1,1)$, $k=1,\dots, m$,  and let $\nu$ be given by
 \begin{align*}
 \nu(\mathrm{d}z) &= \sum_{k=1}^{m} g_k(z_k) \frac{\mathrm{d}z_k}{z_k^{1+\vartheta_k}} \otimes \prod_{j \neq k}\delta_0(\mathrm{d}z_j),
 \end{align*}
 where $g_k: \R_+ \longrightarrow \R_+$ are bounded.
 For each $\xi \geq 1$ we obtain
 \begin{align*}
  \int_{\R_+^m} \left( 1 - \mathrm{e}^{- \xi z_k} \right)\nu(\mathrm{d}z)
  &= \xi^{\vartheta_k} \int_{0}^{\infty}\left( 1 - \mathrm{e}^{- r}\right)g_k\left( \frac{r}{\xi} \right)\frac{\mathrm{d}r}{r^{1+\vartheta_k}}
  \\ &\geq \xi^{\vartheta_k} \int_{0}^{1}\left( 1 - \mathrm{e}^{- r}\right)\frac{\mathrm{d}r}{r^{1+\vartheta_k}} \cdot \inf_{x \in [0,1]}\{g_k(x)\}
 \end{align*}
 Hence condition (A) is satisfied, provided
 $\inf_{x \in [0,1]}\{g_k(x) \}> 0$ holds for all $k = 1,\dots, m$.
\end{Example}

\section{Regularity of the heat kernel for the one-dimensional stable JCIR
process}

In this section we suppose that the conditions of Theorem \ref{th: regularity in one dimension}
are satisfied. Letting $f(t,u)=\Re(\psi(t,u))$ and $g(t,u)=\Im(\psi(t,u))$,
where $\psi$ is obtained from \eqref{one dimensional riccati equation},
we find that $f(t,\mathrm{i}y),g(t,\mathrm{i}y)$ are the unique solutions
to
\[
\begin{cases}
\partial_{t}f=\beta f+\int_{0}^{\infty}\left(\mathrm{e}^{fz}\cos(gz)-1-fz\right)\mu_{\alpha}(\mathrm{d}z), & f(0,\mathrm{i}y)=0,\\
\partial_{t}g=\beta g+\int_{0}^{\infty}\left(\mathrm{e}^{fz}\sin(gz)-gz\right)\mu_{\alpha}(\mathrm{d}z), & g(0,\mathrm{i}y)=y.
\end{cases}
\]
It follows from the general theory of affine processes (see \cite[Theorem 2.7]{DFS03})
that $f\leq0$. This property will be frequently used. The following
is our crucial estimate.
\begin{Proposition}\label{prop:00}
For each $t_{0}>0$, there exist constants $M,C_{1},C_{2}>0$, which depend
on $t_{0}$, such that for all $|y|\geq M$ and $t\ge t_{0}$,
\begin{align}
b\int_{0}^{t}f(s,\mathrm{i}y)\mathrm{d}s+\int_{0}^{t}\int_{0}^{\infty}\left(\mathrm{e}^{zf(s,\mathrm{i}y)}-1\right)\nu(\mathrm{d}z)\leq-C_{1}|y|^{1+\vartheta-\alpha}+C_{2}.\label{eq:00}
\end{align}
\end{Proposition}
Below we first prove Theorem \ref{th: regularity in one dimension}
and then Proposition \ref{prop:00}.
\begin{proof}[Proof of Theorem \ref{th: regularity in one dimension}]
Let $t>0$ be fixed and choose $t_0 \in (0,t)$.
Note that for $u\in\mathbb{R}$, we have $f(t,\mathrm{i}u)\le0$ and
\begin{align}
\Re(\phi(t,\mathrm{i}u)) & =b\int_{0}^{t}f(s,\mathrm{i}u)ds+\int_{0}^{t}\int_{0}^{\infty}\left(\mathrm{e}^{zf(s,\mathrm{i}u)}\cos(zg(s,\mathrm{i}u))-1\right)\nu(\mathrm{d}z)\nonumber \\
 & \le b\int_{0}^{t}f(s,\mathrm{i}u)ds+\int_{0}^{t}\int_{0}^{\infty}\left(\mathrm{e}^{zf(s,\mathrm{i}u)}-1\right)\nu(\mathrm{d}z).\label{eq: real part for psi}
\end{align}
By (\ref{eq: real part for psi}) and Proposition \ref{prop:00},
there exist constants $M,C_{1},C_{2}>0$ such that for all $|u|\ge M$
and $t\ge t_{0}$,
\begin{align}
\left|\E[\mathrm{e}^{\mathrm{i}uX^{x}(t)}]\right| & =\left|\mathrm{e}^{\phi(t,\mathrm{i}u)+x\psi(t,\mathrm{i}u)}\right|\nonumber \\
 & =\mathrm{e}^{\Re(\phi(t,\mathrm{i}u))}\mathrm{e}^{xf(t,\mathrm{i}u)}\nonumber \\
 & \le\mathrm{e}^{\Re(\phi(t,\mathrm{i}u))}\le\exp\left\{ -C_{1}|u|^{1+\vartheta-\alpha}+C_{2}\right\} .\label{eq: uniform esti for char func}
\end{align}
Hence
\begin{align*}
\int_{-\infty}^{\infty}|u|^{p}\left|\E[\mathrm{e}^{\mathrm{i}uX^{x}(t)}]\right|\mathrm{d}u & <\infty
\end{align*}
for all $p\geq0$. So $P_{t}(x,\mathrm{d}y)=p_{t}(x,y)\mathrm{d}y$,
where $p_{t}(x,y)$ is given by
\begin{align}
p_{t}(x,y) & =\frac{1}{2\pi}\int_{-\infty}^{\infty}\mathrm{e}^{-\mathrm{i}yu}\E[\mathrm{e}^{\mathrm{i}uX^{x}(t)}]\mathrm{d}u\nonumber \\
 & =\frac{1}{2\pi}\int_{-\infty}^{\infty}\mathrm{e}^{-\mathrm{i}yu}\mathrm{e}^{\phi(t,\mathrm{i}u)+x\psi(t,\mathrm{i}u)}\mathrm{d}u.\label{eq: respresetnation for pt}
\end{align}
It is clear that the integrand in (\ref{eq: respresetnation for pt})
is jointly continuous in $(t,x,y)\in(0,\infty)\times[0,\infty)^{2}$,
and in view of the estimate (\ref{eq: uniform esti for char func})
we may apply dominated convergence to find that $p_{t}(x,y)$ is also
jointly continuous in $(t,x,y)$.
Using formula (6.16) in the proof of  \cite[Proposition 6.1]{DFS03} we find a constant $C = C_t > 0$ such that $|\psi(t,\mathrm{i}u)| \leq C(1+|u|)$, $u \in \mathbb{R}$.
Hence using \eqref{eq: uniform esti for char func} we may differentiate under the integral in \eqref{eq: respresetnation for pt} and find that $(x,y) \longmapsto p_{t}(x,y)$ is smooth with all derivatives being  bounded. The assertion is proved.
\end{proof}
The rest of this section is devoted to the
proof of Proposition \ref{prop:00}. For the proof we use some ideas
taken from \cite{FMS13}. Namely, for $y\in\R$ with $|y|\neq0$,
introduce
\[
\begin{cases}
F(t,y):=\frac{1}{|y|}f\left(\frac{t}{|y|^{\alpha-1}},\mathrm{i}y\right), & t\ge0,\\
G(t,y):=\frac{1}{|y|}g\left(\frac{t}{|y|^{\alpha-1}},\mathrm{i}y\right), & t\ge0.
\end{cases}
\]
Using the substitution $z\longmapsto|y|z$ shows that $(F,G)$ solve
\[
\begin{cases}
\partial_{t}F=\beta\frac{F}{|y|^{\alpha-1}}+\int_{0}^{\infty}\left(\mathrm{e}^{Fz}\cos(Gz)-1-Fz\right)\mu_{\alpha}(\mathrm{d}z), & F(0,y)=0,\\
\partial_{t}G=\beta\frac{G}{|y|^{\alpha-1}}+\int_{0}^{\infty}\left(\mathrm{e}^{Fz}\sin(Gz)-Gz\right)\mu_{\alpha}(\mathrm{d}z), & G(0,y)=\frac{y}{|y|}.
\end{cases}
\]
We first prove the following lemma. \begin{lem} \label{lem: small time est for F}
There exist constants $t_{0},\delta>0$ and $M>1$ such that for all
$t\in(0,t_{0}]$ and $|y|\ge M$,
\[
F(t,y)\le-\delta t.
\]
\end{lem} \begin{proof} Note that $\partial_{t}F(0,y)=\int_{0}^{\infty}\left(\cos(z)-1\right)\mu_{\alpha}(\mathrm{d}z)<0$
and
\[
\partial_{t}G(0,y)=\begin{cases}
\frac{\beta}{|y|^{\alpha-1}}+\int_{0}^{\infty}\left(\sin(z)-z\right)\mu_{\alpha}(\mathrm{d}z), & y>0,\\
-\frac{\beta}{|y|^{\alpha-1}}-\int_{0}^{\infty}\left(\sin(z)-z\right)\mu_{\alpha}(\mathrm{d}z), & y<0.
\end{cases}
\]
Without loss of generality we suppose $y>0$, which implies $G(0,y)=1$.

By continuity, we find $a>0$ small enough and $M>0$ large enough
such that for all $(F,G)\in D=[-a,0]\times[1-a,1+a]$ and all $|y|\ge M$,
\begin{equation}
-2\delta\le\beta\frac{F}{|y|^{\alpha-1}}+\int_{0}^{\infty}\left(\mathrm{e}^{Fz}\cos(Gz)-1-Fz\right)\mu_{\alpha}(\mathrm{d}z)\le-\delta\label{eq: esti for F}
\end{equation}
and
\begin{equation}
\left|\beta\frac{G}{|y|^{\alpha-1}}+\int_{0}^{\infty}\left(\mathrm{e}^{Fz}\sin(Gz)-Gz\right)\mu_{\alpha}(\mathrm{d}z)\right|\le K,\label{eq: esti for G}
\end{equation}
where $\delta,K>0$ are constants.

Starting from $(0,1)$, the solution $(F(t,y),G(t,y))$ will stay
within $D$ for some positive time, since the velocity vector field
is bounded in $D$. More precisely, let
\[
t_{0}:=\frac{a}{\sqrt{4\delta^{2}+K^{2}}}>0,
\]
then \eqref{eq: esti for F} and \eqref{eq: esti for G} imply that
for $t\in(0,t_{0}]$ and $|y|\ge M$,
\[
\left(F(t,y),G(t,y)\right)\in D
\]
and thus
\[
F(t,y)=\int_{0}^{t}\partial_{s}F(s,y)\mathrm{d}s\le-\int_{0}^{t}\delta\mathrm{d}s=-\delta t.
\]
The lemma is proved. \end{proof}

We are now prepared to provide a full proof of Proposition \ref{prop:00}.
\begin{proof}[Proof of Proposition \ref{prop:00}] Let $T>1$
be such that $T^{-1}<t_{0}<T$. In the following we first prove that
there exist constants $K,C_{1},C_{2}>0$ such that for all $|y|\geq K$
and $t\in[T^{-1},T]$,
\begin{equation}
b\int_{0}^{t}f(s,\mathrm{i}y)\mathrm{d}s+\int_{0}^{t}\int_{0}^{\infty}\left(\mathrm{e}^{zf(s,\mathrm{i}y)}-1\right)\nu(\mathrm{d}z)\leq-C_{1}|y|^{1+\vartheta-\alpha}+C_{2}.\label{eq: second uniform esti for phi}
\end{equation}
Define
\[
\tilde{\beta}:=\begin{cases}
\beta, & \mbox{if \ }\beta<0,\\
-1, & \mbox{if \ }\beta\ge0.
\end{cases}
\]
Using that $\cos(Gz)\leq1$ combined with
\[
\int_{0}^{\infty}\left(e^{Fz}-1-Fz\right)\mu_{\alpha}(\mathrm{d}z)=(-F)^{\alpha},
\]
we find that $F(s,y)$ satisfies
\begin{equation}
\begin{cases}
\partial_{s}F\le\tilde{\beta}\frac{F}{|y|^{\alpha-1}}+(-F)^{\alpha}, & s\ge t_{1},\\
F(t_{1},y)\le-\rho,
\end{cases}\label{eq: comparison for F}
\end{equation}
for all $|y|\ge M>1$. Here $t_{1},\rho,M>0$ are constants whose
existence is guaranteed by Lemma \ref{lem: small time est for F},
and $t_{1}$ can actually be made arbitrarily small such that
\begin{equation}
t_{1}<T^{-1}.\label{eq: choice of t_1}
\end{equation}
Since for $\varkappa\in\R$ the solution to
\[
\partial_{s}\bar{F}=\varkappa\bar{F}+(-\bar{F})^{\alpha},\quad\bar{F}(0)=-\rho
\]
is given by
\[
\bar{F}(s)=-\left(\left(\rho^{1-\alpha}-\varkappa^{-1}\right)e^{-\varkappa(\alpha-1)s}+\varkappa^{-1}\right)^{\frac{1}{1-\alpha}},
\]
by comparison theorem for 1-dimensional ODEs, we obtain
\[
F(s,y)\le-\left(\left(\rho^{1-\alpha}-\frac{|y|^{\alpha-1}}{\tilde{\beta}}\right)\exp\left(\frac{\tilde{\beta}(1-\alpha)}{|y|^{\alpha-1}}\left(s-t_{1}\right)\right)+\frac{|y|^{\alpha-1}}{\tilde{\beta}}\right)^{\frac{1}{1-\alpha}},\quad s\ge t_{1}.
\]
So
\begin{align}
 & f(s,\mathrm{i}y)\nonumber \\
 & \quad=|y|F(|y|^{\alpha-1}s,y)\nonumber \\
 & \quad\le-|y|\left(\left(\rho^{1-\alpha}-\frac{|y|^{\alpha-1}}{\tilde{\beta}}\right)\exp\left(\tilde{\beta}(1-\alpha)\left(s-\frac{t_{1}}{|y|^{\alpha-1}}\right)\right)+\frac{|y|^{\alpha-1}}{\tilde{\beta}}\right)^{\frac{1}{1-\alpha}},\label{eq:esti f(s,iy)}
\end{align}
whenever $s\ge t_{1}/|y|^{\alpha-1}$.

``Case 1'': Suppose $b>0$. Without loss of generality assume $b=1$.
Note that \eqref{eq: choice of t_1} holds. For $|y|\ge M>1$ and
$t\in[T^{-1},T]$, we have
\begin{align*}
 & \int_{0}^{t}f(s,\mathrm{i}y)ds\\
 & \quad\le\int_{t_{1}/|y|^{\alpha-1}}^{t}f(s,\mathrm{i}y)\mathrm{d}s\\
 & \quad\le-|y|\int_{\frac{t_{1}}{|y|^{\alpha-1}}}^{t}\left(\left(\rho^{1-\alpha}-\frac{|y|^{\alpha-1}}{\tilde{\beta}}\right)\exp\left(\tilde{\beta}(1-\alpha)\left(s-\frac{t_{1}}{|y|^{\alpha-1}}\right)\right)+\frac{|y|^{\alpha-1}}{\tilde{\beta}}\right)^{\frac{1}{1-\alpha}}\mathrm{d}s.
\end{align*}
Set
\begin{equation}
\eta(s):=\left(\left(\rho^{1-\alpha}-\frac{|y|^{\alpha-1}}{\tilde{\beta}}\right)\exp\left(\tilde{\beta}(1-\alpha)\left(s-\frac{t_{1}}{|y|^{\alpha-1}}\right)\right)+\frac{|y|^{\alpha-1}}{\tilde{\beta}}\right)^{\frac{1}{1-\alpha}}>0.\label{eq: defi of eta}
\end{equation}
Then
\begin{align*}
\eta'(s) & =\tilde{\beta}\left(\rho^{1-\alpha}-\frac{|y|^{\alpha-1}}{\tilde{\beta}}\right)\exp\left(\tilde{\beta}(1-\alpha)\left(s-\frac{t_{1}}{|y|^{\alpha-1}}\right)\right)\\
 & \qquad\quad\cdot\left(\left(\rho^{1-\alpha}-\frac{|y|^{\alpha-1}}{\tilde{\beta}}\right)\exp\left(\tilde{\beta}(1-\alpha)\left(s-\frac{t_{1}}{|y|^{\alpha-1}}\right)\right)+\frac{|y|^{\alpha-1}}{\tilde{\beta}}\right)^{\frac{\alpha}{1-\alpha}}\\
 & =\tilde{\beta}\eta(s)^{\alpha}\left(\eta(s)^{1-\alpha}-\frac{|y|^{\alpha-1}}{\tilde{\beta}}\right).
\end{align*}
Therefore, substituting $s\to\eta(s)=z$ yields
\begin{align}
\int_{0}^{t}f(s,\mathrm{i}y)\mathrm{d}s & \le-|y|\int_{\eta(t_{1}/|y|^{\alpha-1})}^{\eta(t)}\frac{z}{\tilde{\beta}z^{\alpha}\left(z^{1-\alpha}-\frac{|y|^{\alpha-1}}{\tilde{\beta}}\right)}\mathrm{d}z\nonumber \\
 & =-|y|\int_{\rho}^{\eta(t)}\frac{1}{\left(\tilde{\beta}-|y|^{\alpha-1}z^{\alpha-1}\right)}\mathrm{d}z\nonumber \\
 & =-\int_{|y|\eta(t)}^{|y|\rho}\frac{1}{z^{\alpha-1}-\tilde{\beta}}\mathrm{d}z\le-\frac{|y|\left(\rho-\eta(t)\right)}{\left(|y|\rho\right)^{\alpha-1}-\tilde{\beta}}.\label{eq: first esti for integral f(s)}
\end{align}
Note that $\tilde{\beta}<0$ and the function
\begin{align*}
 & \left(\left(\rho^{1-\alpha}-\frac{r^{\alpha-1}}{\tilde{\beta}}\right)\exp\left(\tilde{\beta}(1-\alpha)\left(t-\frac{t_{1}}{r^{\alpha-1}}\right)\right)+\frac{r^{\alpha-1}}{\tilde{\beta}}\right)^{\frac{1}{1-\alpha}}\\
 & \quad=\left(\rho^{1-\alpha}\exp\left(\tilde{\beta}(1-\alpha)\left(t-\frac{t_{1}}{r^{\alpha-1}}\right)\right)-\frac{r^{\alpha-1}}{\tilde{\beta}}\left(\exp\left(\tilde{\beta}(1-\alpha)\left(t-\frac{t_{1}}{r^{\alpha-1}}\right)\right)-1\right)\right)^{\frac{1}{1-\alpha}}
\end{align*}
is monotone increasing in $r\in(0,\infty)$. Therefore, for $|y|\ge M$
and $t\in[T^{-1},T]$, we obtain
\begin{align}
\rho & \ge\rho-\eta(t)\nonumber \\
 & \ge\rho-\left(\left(\rho^{1-\alpha}-\frac{M^{\alpha-1}}{\tilde{\beta}}\right)\exp\left(\tilde{\beta}(1-\alpha)\left(t-\frac{t_{1}}{M^{\alpha-1}}\right)\right)+\frac{M^{\alpha-1}}{\tilde{\beta}}\right)^{\frac{1}{1-\alpha}}\nonumber \\
 & \ge\rho-\left(\left(\rho^{1-\alpha}-\frac{M^{\alpha-1}}{\tilde{\beta}}\right)\exp\left(\tilde{\beta}(1-\alpha)\left(T^{-1}-\frac{t_{1}}{M^{\alpha-1}}\right)\right)+\frac{M^{\alpha-1}}{\tilde{\beta}}\right)^{\frac{1}{1-\alpha}}\nonumber \\
 & =:c_{1}>0.\label{eq: define c_1}
\end{align}
Combining \eqref{eq: first esti for integral f(s)} and \eqref{eq: define c_1}
gives
\begin{align*}
\int_{0}^{t}f(s,\mathrm{i}y)\mathrm{d}s & \le-\frac{c_{1}|y|}{\left(|y|\rho\right)^{\alpha-1}-\tilde{\beta}}.
\end{align*}
Obviously, we can choose a larger $M'>M$ such that
\[
\int_{0}^{t}f(s,\mathrm{i}y)\mathrm{d}s\le-\frac{c_{1}|y|}{2\left(|y|\rho\right)^{\alpha-1}}\le-c_{2}|y|^{2-\alpha},\quad\forall\ |y|\ge M',\ t\in[T^{-1},T].
\]

``Case 2'': Suppose $b=0$. Similarly as in Case 1, by (\ref{eq: condition for immigraion}), (\ref{eq: real part for psi}) and (\ref{eq:esti f(s,iy)}), we obtain,
for $|y|\ge M>1$ and $t\in[T^{-1},T]$,
\begin{align*}
 & \int_{0}^{t}\int_{0}^{\infty}\left(e^{zf(s,\mathrm{i}y)}-1\right)\nu(\mathrm{d}z)\textrm{d}s\\
 & \quad=-\int_{0}^{t}\int_{0}^{\infty}\left(1-e^{zf(s,\mathrm{i}y)}\right)\nu(\mathrm{d}z)\textrm{d}s\\
 & \quad\le-\int_{\frac{t_{1}}{|y|^{\alpha-1}}}^{t}\left[c_{3}\left(-f(s,\mathrm{i}y)\right)^{\vartheta}-c_{4}\right]\mathrm{d}s\\
 & \quad\le c_{4}t-c_{3}|y|^{\vartheta}\int_{\frac{t_{1}}{|y|^{\alpha-1}}}^{t}\left(\left(\rho^{1-\alpha}-\frac{|y|^{\alpha-1}}{\tilde{\beta}}\right)\exp\left(\tilde{\beta}(1-\alpha)\left(s-\frac{t_{1}}{|y|^{\alpha-1}}\right)\right)+\frac{|y|^{\alpha-1}}{\tilde{\beta}}\right)^{\frac{\vartheta}{1-\alpha}}\mathrm{d}s,
\end{align*}
where we have used \eqref{eq: condition for immigraion} so that
$\int_0^{\infty}(1 - \mathrm{e}^{-\xi z})\nu(\mathrm{d}z) \geq c_3\xi^{\vartheta} - c_4$
for all $\xi \geq 0$ and some constants $c_3,c_4 > 0$.
Using again the change of variables $z=\eta(s)$, where $\eta$ is
defined in \eqref{eq: defi of eta}, we get
\begin{align*}
 & \int_{0}^{t}\int_{0}^{\infty}\left(e^{zf(s,\mathrm{i}y)}-1\right)\nu(\mathrm{d}z)\textrm{d}s\\
 & \quad\le c_{4}t-c_{3}\int_{|y|\eta(t)}^{|y|\rho}\frac{z^{\vartheta-1}}{z^{\alpha-1}-\tilde{\beta}}\mathrm{d}z\\
 & \quad\le c_{4}T-\frac{c_{3}|y|\left(|y|\rho\right)^{\vartheta-1}}{\left(|y|\rho\right)^{\alpha-1}-\tilde{\beta}}\left(\rho-\eta(t)\right)\\
 & \quad\overset{(\ref{eq: define c_1})}{\le}c_{4}T-\frac{c_{1}c_{3}|y|\left(|y|\rho\right)^{\vartheta-1}}{\left(|y|\rho\right)^{\alpha-1}-\tilde{\beta}}\le-c_{5}|y|^{1+\vartheta-\alpha}+c_{6},\quad\forall\ |y|\ge M'',\ t\in[T^{-1},T],
\end{align*}
where $M''>M$ is another large enough constant.

Summarizing Case 1 and 2, and noting that $2-\alpha\ge1+\vartheta-\alpha$,
we obtain \eqref{eq: second uniform esti for phi}. Now, for $t>T$
and $|y|\ge K$, it holds also that
\begin{align*}
 & b\int_{0}^{t}f(s,\mathrm{i}y)\mathrm{d}s+\int_{0}^{t}\int_{0}^{\infty}\left(\mathrm{e}^{zf(s,\mathrm{i}y)}-1\right)\nu(\mathrm{d}z)\\
 & \quad\le b\int_{0}^{T}f(s,\mathrm{i}y)\mathrm{d}s+\int_{0}^{T}\int_{0}^{\infty}\left(\mathrm{e}^{zf(s,\mathrm{i}y)}-1\right)\nu(\mathrm{d}z)\\
 & \quad\le-C_{1}|y|^{1+\vartheta-\alpha}+C_{2}.
\end{align*}
The proposition is proved. \end{proof}

\section{Existence and regularity of the heat kernel}

\subsection{Heat kernel and anisotropic Besov regularity}

In order to measure anisotropic smoothness related to the cylindrical
L\'evy process $Z=(Z_{1},\dots,Z_{m})$, we use an anisotropic analogue
of classical Besov spaces. Corresponding to the regularity indices
$(\alpha_{1},\dots,\alpha_{m})$ we define a mean order of smoothness
$\overline{\alpha}>0$ and an anisotropy $a=(a_{1},\dots,a_{m})$
by
\begin{align}
\frac{1}{\overline{\alpha}}=\frac{1}{m}\left(\frac{1}{\alpha_{1}}+\dots+\frac{1}{\alpha_{m}}\right),\qquad a_{i}=\frac{\overline{\alpha}}{\alpha_{i}},\ \ i=1,\dots,m.\label{MAIN:04}
\end{align}
Then note that $0<a_{1},\dots,a_{m}<\infty$ and $a_{1}+\dots+a_{m}=m$.
Take $\lambda>0$ with $\lambda/a_{k}\in(0,1)$ for all $k\in\{1,\dots,m\}$.
For a measurable function $f:\R^{m}\longrightarrow\R$ introduce
\begin{align}
\|f\|_{B_{1,\infty}^{\lambda,a}}:=\|f\|_{L^{1}(\R^{m})}+\sum\limits _{k=1}^{m}\sup\limits _{h\in[-1,1]}|h|^{-\lambda/a_{k}}\|\Delta_{he_{k}}f\|_{L^{1}(\R^{m})},\label{EQ:36}
\end{align}
where $\Delta_{h}f(x)=f(x+h)-f(x)$, $h\in\R^{m}$. The anisotropic
Besov space $B_{1,\infty}^{\lambda,a}(\R^{m})$ is defined as the
set of all $L^{1}(\R^{m})$ functions $f$ with $\|f\|_{B_{1,\infty}^{\lambda,a}}<\infty$
(see \cite{D03} and \cite{T06} for additional details and references).
By studying estimates on the heat kernel weighted by
\[
\rho_{\delta}(x)=\min\{\delta,x_{1}^{1/\alpha_{1}},\dots,x_{m}^{1/\alpha_{m}}\}\mathbbm{1}_{\R_{+}^{m}}(x),\qquad\delta\in(0,1],
\]
we can also take the behavior of the process at the boundary into
account. The following is our main result for the regularity of the heat kernel.
\begin{Theorem}\label{th:04}
Suppose that condition (A) is satisfied
and assume there exists $\tau>0$ satisfying
\[
\int_{\R_{+}^{m}}\mathbbm{1}_{\{|z|>1\}}|z|^{1+\tau}\nu(\mathrm{d}z)<\infty.
\]
Then for each $t>0$ and each $x\in\R_{+}^{m}$ the transition kernel
$P_{t}(x,\mathrm{d}y)$ has density $p_{t}(x,y)$ with respect to
the Lebesgue measure. Moreover, there exists some
small constant $\lambda>0$ such that for each $T>0$, $\varkappa\in(0,1]$
and $\delta\in(0,1]$,
\begin{equation}
\|p_{t}^{\delta}(x,\cdot)\|_{B_{1,\infty}^{\lambda,a}(\R_{+}^{m})}\leq C(1+|x|)^{\varkappa}(1\wedge t)^{-1/\alpha_{\mathrm{min}}},\quad t\in(0,T],\ x\in\R_{+}^{m},\label{EQA:00}
\end{equation}
where $p_{t}^{\delta}(x,y):=\rho_{\delta}(y)p_{t}(x,y)$ and $C=C(\lambda,\tau,\varkappa,\delta, T)>0$
is a constant.
\end{Theorem}
The proof of this result follows the
arguments given in \cite{FJR18,FJR18a} where general stochastic equations
have been considered. Since we need the precise dependence on $x$
in \eqref{EQA:00} and since the proofs are significantly simpler
for \eqref{eq: a roots model} compared with the general case, we
provide, for convenience of the reader, a full proof of Proposition
\ref{th:04} in the appendix.

\subsection{Boundary non-attainment}

In this section we prove the following estimate on the behavior of
the anisotropic stable JCIR process at the boundary. \begin{Proposition}\label{th:boundary behavior}
Suppose that condition (A) is satisfied. Then for each $t>0$
and each $R>0$, there exists $C>0$ such that
\[
\P[\min\{X_{1}^{x}(t),\dots,X_{m}^{x}(t)\}\leq\e]\leq C\e, \qquad \e\in(0,1),\ x\in\R_{+}^{m},\ |x|\leq R.
\]
In particular, $\P[X^{x}(t)\in\R_{++}^{m}]=1$ holds for all $t>0$
and all $x\in\R_{+}^{m}$. \end{Proposition} \begin{proof} Consider
$x=(x_{1},\dots,x_{m})\in\R_{+}^{m}$ and let $(X^{x}(t))_{t\geq0}$
be the anisotropic stable JCIR process obtained from \eqref{eq: a roots model}.
Moreover, let $Y^{x}(t)=(Y_{1}^{x_{1}}(t),\dots,Y_{m}^{x_{m}}(t))$
be the unique $\R_{+}^{m}$-valued strong solution to
\begin{align}
\mathrm{d}Y_{k}^{x_{k}}(t) & =\left(b_{k}+\beta_{kk}Y_{k}^{x_{k}}(t)\right)\mathrm{d}t+\sqrt[\alpha_{k}]{\sigma_{k}Y_{k}^{x_{k}}(t)}\mathrm{d}Z_{k}(t)+\mathrm{d}J_{k}(t),\label{eq:12}
\end{align}
where $k\in\{1,\dots,d\}$ and $Y_{k}^{x_{k}}(0)=x_{k}$. Existence
and uniqueness of such a process is again a direct consequence of
\cite{BLP15}, see also \cite{FL10}. Moreover, $(Y^{x}(t))_{t\geq0}$
is the anisotropic JCIR process where all off-diagonal drift terms
equal to zero. Using the fact that $\beta_{kj}\geq0$ whenever $k\neq j$
we may apply the comparison result established in \cite[Proposition 4.2]{FJR19b}
to deduce
\[
\P[X_{k}^{x}(t)\geq Y_{k}^{x_{k}}(t),\ \ t\geq0]=1,\qquad k\in\{1,\dots,m\}.
\]
Since $(J_{k}(t))_{t\geq0}$ is a L\'evy subordinator on $\R_{+}$ whose
L\'evy measure is given by $\nu_{k}=\nu\circ\mathrm{pr}_{k}^{-1}$,
where $\mathrm{pr}_{k}(z)=z_{k}$ denotes the projection on the $k$-th
coordinate, we can apply Theorem \ref{th: regularity in one dimension}
for the process $(Y_{k}^{x_{k}}(t))_{t\geq0}$. Let $p_{t}^{k}(x_{k},y_{k})$
be its heat kernel. Then
\begin{align*}
\P[\min\{X_{1}^{x}(t),\dots,X_{m}^{x}(t)\}\leq\e] & \leq\sum_{k=1}^{m}\P[X_{k}^{x}(t)\leq\e]\\
 & \leq\sum_{k=1}^{m}\P[Y_{k}^{x_{k}}(t)\leq\e]\\
 & =\sum_{k=1}^{m}\int_{0}^{\e}p_{t}^{k}(x_{k},y_{k})\mathrm{d}y_{k}\leq C\e,
\end{align*}
since $p_{t}^{k}(x_{k},y_{k})$ is jointly continuous in $(x_{k},y_{k})$
and $|x|\leq R$. This proves the assertion. \end{proof}

\subsection{Proof of Theorem \ref{corr L1 continuity}}

Let us first prove a slightly weaker assertion.
\begin{Lemma} \label{LEMMA:03}
Assume the same assumptions as in Theorem \ref{th:04}.
Then the mapping $\R_{+}^{m}\ni x\longmapsto p_{t}^{\delta}(x,\cdot)\in L^{1}(\R_{+}^{m})$
is continuous for each $t>0$ and each $\delta\in(0,1]$, where $p_{t}^{\delta}(x,y)=\rho_{\delta}(y)p_{t}(x,y)$.
\end{Lemma}
\begin{proof}
Using the Feller property (see \cite[Proposition 8.2]{DFS03}),
we find that $x\longmapsto p_{t}(x,y)\mathrm{d}y$ is weakly continuous,
i.e.
\begin{align}
x\longmapsto\int_{\R_{+}^{m}}f(y)\rho_{\delta}(y)p_{t}(x,y)\mathrm{d}y\label{EQA:01}
\end{align}
is continuous for each bounded continuous function $f$.

Let $x\in\R_{+}^{m}$ be fixed. Suppose $(x_{n})_{n}$ is a sequence
such that $x_{n}\to x$. Using \eqref{EQA:00} we find that
\[
\sup_{n\in\mathbb{N}}\|p_{t}^{\delta}(x_{n},\cdot)\|_{B_{1,\infty}^{\lambda,a}}<\infty.
\]
Next observe that for each $K,R>0$,
\[
\sup\limits _{|x|\leq K}\int_{|y|\geq R}p_{t}^{\delta}(x,y)\mathrm{d}y\leq \delta\sup\limits _{|x|\leq K}\int_{|y|\geq R}p_{t}(x,y)\mathrm{d}y\leq\frac{\delta}{R}\sup\limits _{|x|\leq K}\E[|X_{t}^{x}|]\leq\frac{\delta KC_{t}}{R},
\]
where we have used Proposition \ref{moment estimate}. Hence we may
apply the Kolmogorov-Riesz compactness criterion which gives existence
of a subsequence $(x_{n_{k}})_{k}$ such that $p_{t}^{\delta}(x_{n_{k}},\cdot)$
has a limit in $L^{1}(\R_{+}^{m})$. By the weak continuity \eqref{EQA:01}
this limit is exactly $p_{t}^{\delta}(x,\cdot)$.

So for each sequence $(x_{n})_{n}$ with $x_{n}\to x$
we have found a subsequence $(x_{n_{k}})_{k}$ such that $p_{t}^{\delta}(x_{n_{k}},\cdot)\to p_{t}^{\delta}(x,\cdot)$
in $L^{1}$ as $k\to\infty$. This proves the desired $L^{1}$ continuity.
\end{proof}
We are now prepared to provide a full proof of Theorem
\ref{corr L1 continuity}. The affine structure of the anisotropic
JCIR process allows us to study first the particular case with
$\nu$ having no big jumps, i.e., $\nu(\mathrm{d}z)=\mathbbm{1}_{\{|z|\le1\}}\nu(\mathrm{d}z)$,
and then the general case by a convolution argument.
Namely, define $\nu_{0}(\mathrm{d}z)=\mathbbm{1}_{\{|z|\le1\}}\nu(\mathrm{d}z)$
and $\nu_{1}(\mathrm{d}z)=\mathbbm{1}_{\{|z|>1\}}\nu(\mathrm{d}z)$,
and write $F(u)=\langle b,u\rangle+F_{0}(u)+F_{1}(u)$ where for $i=0,1$,
\[
F_{i}(u)=\int_{\R_{+}^{m}}\left(\mathrm{e}^{\langle u,z\rangle}-1\right)\nu_{i}(\mathrm{d}z),\qquad u\in\mathbb{C}^{m}  \text{ with } \mathrm{Re}(u) \leq 0.
\]
Let $(Y^{x}(t))_{t\geq0}$ the unique strong solution to \eqref{eq: a roots model}
with $\nu=\nu_{0}$ and let $(\widetilde{Y}^{x}(t))_{t\geq0}$ be
the unique strong solution to \eqref{eq: a roots model} with $b=0$
and $\nu=\nu_{1}$, i.e.,
\begin{align*}
\E\left[\mathrm{e}^{\mathrm{i}\langle u,Y^{x}(t)\rangle}\right] & =\exp\left(\int_{0}^{t}\left(\langle b,\psi(s,\mathrm{i}u)\rangle+F_{0}(\psi(s,\mathrm{i}u))\right)\mathrm{d}s+\langle x,\psi(s,\mathrm{i}u)\rangle\right),\\
\E\left[\mathrm{e}^{\mathrm{i}\langle u,\widetilde{Y}^{x}(t)\rangle}\right] & =\exp\left(\int_{0}^{t}F_{1}(\psi(s,\mathrm{i}u))\mathrm{d}s+\langle x,\psi(s,\mathrm{i}u)\rangle\right),
\end{align*}
where $\psi$ is obtained from \eqref{eq:riccati}. Denote by $Q_{t}^{0}(x,\cdot)$
the transition probabilities of $(Y^{x}(t))_{t\geq0}$ and by $Q_{t}^{1}(x,\cdot)$
the transition probabilities of $(\widetilde{Y}^{x}(t))_{t\geq0}$.
Using \eqref{affine property} we find
\begin{align*}
 & \ \E\left[\mathrm{e}^{\mathrm{i}\langle u,Y^{x}(t)\rangle}\right]\E\left[\mathrm{e}^{\mathrm{i}\langle u,\widetilde{Y}^{0}(t)\rangle}\right]\\
 & =\exp\left(\int_{0}^{t}\left(\langle b,\psi(s,\mathrm{i}u)\rangle+F_{0}(\psi(s,\mathrm{i}u))\right)\mathrm{d}s+\langle x,\psi(s,\mathrm{i}u)\rangle\right)\exp\left(\int_{0}^{t}F_{1}(\psi(s,\mathrm{i}u))\mathrm{d}s\right)\\
 & =\exp\left(\phi(t,\mathrm{i}u)+\langle x,\psi(t,\mathrm{i}u)\rangle\right)\\
 & =\E\left[\mathrm{e}^{\mathrm{i}\langle u,X^{x}(t)\rangle}\right]
\end{align*}
which yields
\begin{equation}
P_{t}(x,\cdot)=Q_{t}^{0}(x,\cdot)\ast Q_{t}^{1}(0,\cdot),\qquad t>0,\ \ x\in\R_{+}^{m},\label{eq: decompose P(t,.)}
\end{equation}
where $\ast$ denotes the convolution of measures.
 \begin{proof}[Proof of Theorem \ref{corr L1 continuity}]
 According to Theorem \ref{th:04}, the kernel $Q_{t}^{0}(x,\cdot)$ has a density
$q_{t}^{0}(x,\cdot)$ for $t>0$. In view of \eqref{eq: decompose P(t,.)}, for $t>0$ and
$x\in\R_{+}^{m}$, the
measure $P_{t}(x,\cdot)$ possesses also a density $p_{t}(x,\cdot)$
with respect to the Lebesgue measure and it is given by
\[
p_{t}(x,y)=\int_{\R_{+}^{m}}q_{t}^{0}(x,y-z)Q_{t}^{1}(0,\mathrm{d}z),\quad y\in\R_{+}^{m},
\]
and $p_{t}(x,y)=0$ for $y\notin\R_{+}^{m}$.

Fix $t>0$ and $x\in\R_{+}^{m}$. Let $(x_{n})_{n\in\mathbb{N}}\subset\R_{+}^{m}$
be such that $x_{n}\to x$. Our aim is to show that
\begin{equation}
p_{t}(x_{n},\cdot)\to p_{t}(x,\cdot)\quad\mbox{in\ensuremath{\quad L^{1}\left(\R_{+}^{m}\right)}},\quad\mbox{as}\quad n\to\infty.\label{eq: to be proved}
\end{equation}
We will finish the proof in two steps:

``Step 1'': We show that $q_{t}^{0}(x_{n},\cdot)$ converges
in $L^{1}\left(\R_{+}^{m}\right)$ to $q_{t}^{0}(x,\cdot)$ as $n\to\infty$. Let $\e\in(0,1)$. Write $\widetilde{\rho}(z)=\min\{z_{1}^{1/\alpha_{1}},\dots,z_{m}^{1/\alpha_{m}}\}$
so that $\rho_{\e}(z)=\min\{\e,\widetilde{\rho}(z)\}$. Then
\begin{align*}
 & \ \|q_{t}^{0}(x_{n},\cdot)-q_{t}^{0}(x,\cdot)\|_{L^{1}\left(\R_{+}^{m}\right)}\\
 & \leq\int_{\R_{+}^{m}}\left(q_{t}^{0}(x_{n},y)+q_{t}^{0}(x,y)\right)\mathbbm{1}_{\{\widetilde{\rho}(y)\leq\e\}}\mathrm{d}y\\
 & \ \ \ +\int_{\R_{+}^{m}}\left|q_{t}^{0}(x_{n},y)-q_{t}^{0}(x,y)\right|\mathbbm{1}_{\{\widetilde{\rho}(y)>\e\}}\mathrm{d}y\\
 & =\P[\widetilde{\rho}(Y^{x_{n}}(t))\leq\e]+\P[\widetilde{\rho}(Y^{x}(t))\leq\e]\\
 & \ \ \ +\int_{\R_{+}^{m}}\left|q_{t}^{0}(x_{n},y)-q_{t}^{0}(x,y)\right|\mathbbm{1}_{\{\widetilde{\rho}(y)>\e\}}\mathrm{d}y.
\end{align*}
Using Proposition \ref{th:boundary behavior} we
can have
\begin{align*}
 & \sup\limits _{n\in\mathbb{N}}\P[\widetilde{\rho}(Y^{x_{n}}(t))\leq\e]+\P[\widetilde{\rho}(Y^{x}(t))\leq\e]\\
 & \quad\le\sup\limits _{n\in\mathbb{N}}\P\left[\bigcup_{i=1}^{m}\left\{ \left(Y_{i}^{x_{n}}(t)\right)^{1/\alpha_{i}}\leq\e\right\} \right]+\P\left[\bigcup_{i=1}^{m}\left\{ \left(Y_{i}^{x}(t)\right)^{1/\alpha_{i}}\leq\e\right\} \right]\\
 & \quad\le\sup\limits _{n\in\mathbb{N}}m\P\left[\min\left\{ Y_{1}^{x_{n}}(t),\dots,Y_{m}^{x_{n}}(t)\right\} \leq\e^{\alpha_{\min}}\right]+m\P\left[\min\left\{ Y_{1}^{x}(t),\dots,Y_{m}^{x}(t)\right\} \leq\e^{\alpha_{\min}}\right]\\
 & \quad\le C\e^{\alpha_{\min}}\le C\e.
\end{align*}
For the third term we use
\begin{align*}
 & \int_{\R_{+}^{m}}\left|q_{t}^{0}(x_{n},y)-q_{t}^{0}(x,y)\right|\mathbbm{1}_{\{\widetilde{\rho}(y)>\e\}}\mathrm{d}y\\
 & \quad=\e^{-1}\int_{\R_{+}^{m}}\left|\rho_{\e}(y)q_{t}^{0}(x_{n},y)-\rho_{\e}(y))q_{t}^{0}(x,y)\right|\mathbbm{1}_{\{\widetilde{\rho}(y)>\e\}}\mathrm{d}y\\
 & \quad\le\e^{-1}\int_{\R_{+}^{m}}\left|\rho_{\e}(y)q_{t}^{0}(x_{n},y)-\rho_{\e}(y))q_{t}^{0}(x,y)\right|\mathrm{d}y,
\end{align*}
where the right-hand side tends by Lemma \ref{LEMMA:03}
to zero as $n\to\infty$. So
\[
\limsup\limits _{n\to\infty}\|q_{t}^{0}(x_{n},\cdot)-q_{t}^{0}(x,\cdot)\|_{L^{1}\left(\R_{+}^{m}\right)}\leq C\e.
\]
Since $\e\in(0,1)$ is arbitrary, the desired convergence $q_{t}^{0}(x_{n},\cdot)\to q_{t}^{0}(x,\cdot)$
in $L^{1}\left(\R_{+}^{m}\right)$ is proved.

``Step 2'': We show that \eqref{eq: to be proved} is true in the
general case. We have
\begin{align*}
 & \int_{\R_{+}^{m}}\left|p_{t}(x_{n},y)-p_{t}(x,y)\right|\mathrm{d}y\\
 & \quad=\int_{\R^{m}}\left|p_{t}(x_{n},y)-p_{t}(x,y)\right|\mathrm{d}y\\
 & \quad\le\int_{\R^{m}}\int_{\R^{m}}\left|q_{t}^{0}(x_{n},y-z)-q_{t}^{0}(x,y-z)\right|Q_{t}^{1}(0,\mathrm{d}z)\mathrm{d}y\\
 & \quad=\int_{\R^{m}}\int_{\R^{m}}\left|q_{t}^{0}(x_{n},y-z)-q_{t}^{0}(x,y-z)\right|\mathrm{d}yQ_{t}^{1}(0,\mathrm{d}z)\\
 & \quad=\|q_{t}^{0}(x_{n},\cdot)-q_{t}^{0}(x,\cdot)\|_{L^{1}\left(\R_{+}^{m}\right)} \to 0, \quad \mbox{as} \quad n\to \infty.
\end{align*}
So \eqref{eq: to be proved} is true. The theorem is proved.
\end{proof}

\section{Exponential ergodicity in total variation}

\subsection{Regularity for the invariant measure}
As a consequence of our regularity results for the heat kernel (see
Sections 3 and 4), we can also deduce similar results for the invariant
measure. We start with the one-dimensional case.
\begin{Corollary}
Suppose that \eqref{eq: condition for immigraion} is satisfied.
Assume that $\beta<0$ and $\int_{(1,\infty)}\log(1+z)\nu(\mathrm{d}z)<\infty$.
Then the unique invariant measure $\pi(\mathrm{d}x)$ has a smooth density $g$ which satisfies $g(x)=0$ for all $x\leq0$
 and vanishes at infinity.
\end{Corollary} \begin{proof} It was shown in \cite{KM12} (see
also \cite{JKR18}) that for all $y\in\R$,
\[
\lim\limits _{t\to\infty}\mathrm{e}^{\phi(t,\mathrm{i}y)+x\psi(t,\mathrm{i}y)}=\mathrm{e}^{\phi(\infty,\mathrm{i}y)}=\int_{\R}\mathrm{e}^{\mathrm{i}xy}\pi(\mathrm{d}x),
\]
where
\[
\phi(\infty,\mathrm{i}y)=\int_{0}^{\infty}\left(b\psi(s,\mathrm{i}y)+\int_{(0,\infty)}\left(\mathrm{e}^{z\psi(s,\mathrm{i}y)}-1\right)\nu(\mathrm{d}z)\right)\mathrm{d}s
\]
and the integral against $\mathrm{d}s$ is absolutely convergent.
Since the process is supported on $\R_{+}$ it is clear that $\pi((-\infty,0))=0$.
Using Proposition \ref{prop:00} we may take the limit $t\to\infty$
in \eqref{eq:00} and find for $t_{0}=1$ constants $M,C>0$
such that
\begin{align*}
\left|\mathrm{e}^{\phi(\infty,\mathrm{i}y)}\right|\leq\mathrm{e}^{-C|y|^{1+\vartheta-\alpha}},\qquad|y|\geq M.
\end{align*}
The assertion follows from classical properties of the Fourier transform.
\end{proof} In the multi-dimensional case we may
use \eqref{EQA:00} to deduce the same regularity for the unique invariant
measure $\pi$.
\begin{Corollary}\label{corr regularity invariant measure}
 Suppose that the anisotropic
stable JCIR process is subcritical, satisfies condition (A) and $\nu$
satisfying
\[
\int_{\R_{+}^{m}}\mathbbm{1}_{\{|z|>1\}}|z|^{1+\tau}\nu(\mathrm{d}z)<\infty
\]
for some $\tau>0$. Then $\pi$ is absolutely continuous with respect
to the Lebesgue measure, i.e., $\pi(\mathrm{d}x)=g(x)\mathrm{d}x$
and there exists a constant $C>0$ such that
\[
\|g^{1}\|_{B_{1,\infty}^{\lambda,a}}\leq C\int_{\R_{+}^{m}}(1+|x|)\pi(\mathrm{d}x)<\infty,
\]
where $g^{1}(x)=\rho_{1}(x)g(x)$ and $\lambda,a$ are given as in
Theorem \ref{th:04}.
\end{Corollary}
\begin{proof}
Using the invariance of $\pi$ we get
\[
\pi(\mathrm{d}x)=\int_{\R_{+}^{m}}P_{t}(y,\mathrm{d}x)\pi(\mathrm{d}y)=\left(\int_{\R_{+}^{m}}p_{t}(y,x)\pi(\mathrm{d}y)\right)\mathrm{d}x,
\]
i.e., $\pi(\mathrm{d}x)$ has a density $g(x)$ which satisfies
\[
g(x)=\int_{\R_{+}^{m}}p_{t}(y,x)\pi(\mathrm{d}y).
\]
Hence we obtain,  for $p_t^1(x,y) = \rho_1(y)p_t(x,y)$,
\begin{align*}
\|\pi^{1}\|_{B_{1,\infty}^{\lambda,a}} & \leq\int_{\R_{+}^{m}}\|p_{t}^{1}(y,\cdot)\|_{B_{1,\infty}^{\lambda,a}}\pi(\mathrm{d}y)\\
 & \leq C\int_{\R_{+}^{m}}(1+|y|)\pi(\mathrm{d}y).
\end{align*}
Using Proposition \ref{uniform boundedness first moment} combined with the weak convergence $p_t(x,y)\mathrm{d}y \longrightarrow \pi(\mathrm{d}y)$ we find that
$\int_{\R_+^m}|y|\pi(\mathrm{d}y) < \infty$.
The assertion is proved.
\end{proof}

\subsection{Proof of Theorem \ref{th:main result}}
Here and below we suppose that the conditions of Theorem \ref{th:main result} are satisfied.
As in Section 4.3, let $\nu_{0}(\mathrm{d}z)=\mathbbm{1}_{\{|z|\le1\}}\nu(\mathrm{d}z)$, $\nu_{1}(\mathrm{d}z)=\mathbbm{1}_{\{|z|>1\}}\nu(\mathrm{d}z)$ and let $(Y^{x}(t))_{t\geq0}$ be the  solution to \eqref{eq: a roots model}
with $\nu=\nu_{0}$ and  $(\widetilde{Y}^{x}(t))_{t\geq0}$
be the  solution to \eqref{eq: a roots model} with $b=0$
and $\nu=\nu_{1}$. Denote by $Q_{t}^{0}(x,\cdot)$
the transition probabilities of $(Y^{x}(t))_{t\geq0}$ and by $Q_{t}^{1}(x,\cdot)$
the transition probabilities of $(\widetilde{Y}^{x}(t))_{t\geq0}$. Recall that \eqref{eq: decompose P(t,.)} holds.
In order to prove Theorem \ref{th:main result}
we first establish a similar statement for $(Y^x(t))_{t \geq 0}$.
\begin{Proposition}
 There exists constants $C, \delta > 0$ such that
\begin{align}\label{eq:13}
    \| Q^0_t(x,\cdot) - Q^0_t(y,\cdot) \|_{\mathrm{TV}}
    \leq C \min\left\{1, ( 1 + |x| + |y|) \mathrm{e}^{- \delta t} \right\},
\end{align}
 for all $x,y \in \R_+^m$ and all $t \geq 0.$
\end{Proposition}
\begin{proof}
Following \cite[Lemma 3.4]{JKR18} we define a new norm by
\begin{align}\label{eq:02}
 |x|_M = \langle x, x \rangle_M^{1/2} = \langle x, M x \rangle^{1/2}
 \ \ \text{ where } M = \int_0^{\infty} \mathrm{e}^{t \beta^{\top}}\mathrm{e}^{t \beta} \mathrm{d}t.
\end{align}
Then note that $M$ is symmetric, positive definite and satisfies
$M\beta + \beta^{\top}M = - \mathbbm{1}$.
In view of Theorem \ref{th:3} it suffices to show that the following two properties are satisfied.
\begin{enumerate}
    \item[(i)] The function $V(x) = (1+|x|_M^2)^{1/2}$ belongs to the domain of the extended generator and there exist constants $c_1,c_2 > 0$ such that
    \[
     L_0V(x) \leq - c_1 V(x) + c_2, \qquad x \in \R_+^m,
    \]
    where $L_0$ denotes the extended generator of $(Y^x(t))_{t \geq 0}$.
    \item[(ii)] For every $R > 0$ there exist $h > 0$ and $\delta \in (0,2)$ with
    \[
      \| Q^0_h(x,\cdot) - Q^0_h(y,\cdot) \|_{\mathrm{TV}} \leq 2 - \delta,
    \]
    for all $x,y \in \R_+^m$ with $|x|,|y| \leq R$.
\end{enumerate}
Property (i) can be shown by similar (but essentially simpler) arguments to
\cite[Lemma 3.4 and Proposition 3.7]{JKR18}. For the sake of completeness a proof is outlined in the appendix, see Lemma \ref{Lyapunov estimate}.
Let us now prove property (ii).
Let $R > 0$ and take any $x,y \in \R_+^m$ with $|x|, \ |y| \leq R$.
Fix any bounded measurable function $f$ on $\R_+^m$ satisfying $\| f\|_{\infty} \leq 1$. Choose $h > 1$ and $\widetilde{R} > 0$ to be specified later on.
Let $H = H_{h,x,y}$ be the joint distribution of $(Y^x(h-1), Y^y(h-1))$, i.e.
$H(\mathrm{d}\widetilde{x}, \mathrm{d}\widetilde{y}) = \P[ Y^x(h-1) \in \mathrm{d}\widetilde{x},\ Y^y(h-1) \in \mathrm{d}\widetilde{y}]$.
Since $(Q_t^0)_{t \geq 0}$ satisfies the conditions of Theorem \ref{corr L1 continuity}, $Q_t^0(x,\cdot)$ has density $q_t^0(x,\cdot)$
being continuous in $x$ with respect to $L^1(\R^m)$.
Hence we find $\eta(\widetilde{R}) > 0$ (independent of $f$)
such that $\|q_1^0(\widetilde{x},\cdot) - q_1^0(\widetilde{y},\cdot)\|_{L^1(\R^m_+)} \leq 1$
for $(\widetilde{x}, \widetilde{y}) \in \Theta$, where
\[
 \Theta = \{ (\widetilde{x}, \widetilde{y}) \in \R_+^m \times \R_+^m \ | \ |\widetilde{x} - \widetilde{y}| \leq \eta(\widetilde{R}), \ |\widetilde{x}|, \ |\widetilde{y}| \leq \widetilde{R} \}.
\]
It follows easily that  $|Q_1^0f(\widetilde{x}) - Q_1^0f(\widetilde{y})| \leq 1$ for $(\widetilde{x}, \widetilde{y}) \in \Theta$. Then we obtain from $\| Q_1^0 f \|_{\infty} \leq 1$
\begin{align*}
    | Q_h^0f(x) - Q_h^0f(y)|
    &\leq \left| \int_{\Theta} (Q_1^0f(\widetilde{x}) - Q_1^0f(\widetilde{y})) H(\mathrm{d}\widetilde{x}, \mathrm{d}\widetilde{y})\right|
    \\ &\ \ \ + \left| \int_{\Theta^c} (Q_1^0f(\widetilde{x}) - Q_1^0f(\widetilde{y})) H(\mathrm{d}\widetilde{x}, \mathrm{d}\widetilde{y})\right|
    \\ &\leq H(\Theta) + 2 H(\Theta^c)
    \\ &= 2 - H(\Theta).
\end{align*}
Next we obtain
\begin{align*}
 H(\Theta^c)&\leq \P[ |Y^x(h-1) - Y^y(h-1)| > \eta(\widetilde{R}) ] + \P[ |Y^x(h-1)| > \widetilde{R} ]
 + \P[ |Y^y(h-1)| > \widetilde{R} ]
 \\ &\leq \eta(\widetilde{R})^{-1}\E[ |Y^x(h-1) - Y^y(h-1) |] + \widetilde{R}^{-1}\E\left[ |Y^x(h-1)| + |Y^y(h-1)| \right]
 \\ &\leq \frac{m}{\eta(\widetilde{R})}|x-y| \mathrm{e}^{- c (h-1)} + \frac{C(1 + |x| + |y|)}{\widetilde{R}}
 \\ &\leq \frac{2m R \mathrm{e}^{c}}{\eta(\widetilde{R})} \mathrm{e}^{-ch} + \frac{C(1 + 2R)}{\widetilde{R}},
\end{align*}
where $c > 0$ and $C > 0$ are some constants given by
\cite[Proposition 6.1]{FJR19a} and Proposition \ref{uniform boundedness first moment}.
Take first $\widetilde{R} > 0$ and then $h > 1$ large enough such that
\[
 \frac{C(1 + 2R)}{\widetilde{R}} < \frac{1}{2} \ \text{ and } \
 \frac{2mR\mathrm{e}^c}{\eta(\widetilde{R})}\mathrm{e}^{-ch} < \frac{1}{2}.
\]
Then
\[
 H(\Theta) = 1 - H(\Theta^c) \geq 1 - \frac{2mR\mathrm{e}^{c}}{\eta(\widetilde{R})} \mathrm{e}^{- c h} - \frac{C(1+ 2R)}{\widetilde{R}} =: \delta \in (0,1).
\]
This proves the assertion.
\end{proof}
We are now prepared to give a proof for Theorem \ref{th:main result}.
\begin{proof}[Proof of Theorem \ref{th:main result}]
Let $\pi$ be the unique invariant measure
and let $H$ be a coupling of $\delta_x$ and $\pi$,
i.e., a Borel probability measure over $\R_+^m \times \R_+^m$ whose marginals are $\delta_x$ and $\pi$, respectively. Then
\begin{align*}
    \| P_t(x, \cdot) - \pi \|_{\mathrm{TV}}
    &\leq \int_{ \R_+^m \times \R_+^m} \| P_t(y,\cdot) - P_t(\widetilde{y},\cdot) \|_{\mathrm{TV}} H(\mathrm{d}y, \mathrm{d}\widetilde{y})
    \\ &\leq \int_{ \R_+^m \times \R_+^m} \| Q^0_t(y,\cdot) - Q^0_t(\widetilde{y},\cdot) \|_{\mathrm{TV}} H(\mathrm{d}y, \mathrm{d}\widetilde{y})
    \\ &\leq C\int_{ \R_+^m \times \R_+^m} \min\left\{ 1, (1+ |y| + |\widetilde{y}|) e^{- \delta t} \right\} H(\mathrm{d}y, \mathrm{d}\widetilde{y}),
\end{align*}
where we have used \cite[Lemma 2.3]{FJKR19} and then \eqref{eq:13} for the integrand.
Note that for $a,b\ge 0$,
\begin{align*}
1\wedge (ab) & \le  C\log(1+ab) \\
 &\leq C\min\{\log(1+a), \log(1+b)\} + C\log(1+a)\log(1+b)
 \\ &\leq C \log(1+a)( 1 + \log(1+b))
 \\ &\leq C a (1 + \log(1+b)),
\end{align*}
where the second inequality is proved in \cite[Lemma 8.5]{FJR19a}. Choosing $a = \mathrm{e}^{- \delta t}$ and $b = 1 + |y| + |\widetilde{y}|$  in the last inequality gives
\begin{align*}
     \| P_t(x, \cdot) - \pi \|_{\mathrm{TV}}
     &\leq C\mathrm{e}^{-\delta t}\int_{ \R_+^m \times \R_+^m} \left(1 + \log(2 + |y| +|\widetilde{y}|)\right)  H(\mathrm{d}y, \mathrm{d}\widetilde{y})
     \\ &\leq C\mathrm{e}^{-\delta t}\int_{ \R_+^m \times \R_+^m} \left(1 + \log(1 + |y|) + \log(1 +|\widetilde{y}|)\right)  H(\mathrm{d}y, \mathrm{d}\widetilde{y})
     \\ &= C \mathrm{e}^{- \delta t} \left( 1 + \log(1 + |x|) + \int_{\R_+^m}\log(1 + |y|) \pi(\mathrm{d}y) \right),
\end{align*}
where we have used the subadditivity $\log(1 + a + b) \leq \log(1 + a) + \log(1 + b)$.
This completes the proof of Theorem \ref{th:main result}.
\end{proof}

\appendix

\section{Moments of the anisotropic stable JCIR process}

The following can be shown by rather standard arguments, see
e.g. \cite[Proposition 5.1]{FJR19a}.
\begin{Proposition}\label{moment estimate}
Let $\eta\in(0,\alpha_{\min})$ and suppose that
\begin{align}
\int_{|z|>1}|z|^{\eta}\nu(\mathrm{d}z)<\infty.\label{eq:10}
\end{align}
Then for each $T>0$, there exists a constant $C_{T}>0$ such that
\[
\sup_{t\in[0,T]}\E[|X^{x}(t)|^{\eta}]\leq C_{T}(1+|x|)^{\eta},\qquad x\in\R_{+}^{m}.
\]
\end{Proposition}
In particular, if \eqref{eq:10} holds for $\eta=1$,
then $(X(t))_{t\geq0}$ has finite first moment. This moment was computed
in \cite{BLP15} where it was shown that
\begin{align}
\E[X^{x}(t)]=\mathrm{e}^{\beta t}x+\int_{0}^{t}\mathrm{e}^{\beta s}\left(b+\int_{\R_{+}^{m}}z\nu(\mathrm{d}z)\right)\mathrm{d}s.\label{eq:first moment}
\end{align}
Actually in \cite{BLP15} the more general class of multi-type continuous-state
branching processes were studied. An extension of such a formula to
general affine processes on the canonical state space was obtained
in \cite{FJR19a}. One simple consequence is the uniform boundedness
of the first moment stated below.
\begin{Proposition}\label{uniform boundedness first moment}
Suppose that $\beta$ has only eigenvalues with negative real-parts
and \eqref{eq:10} holds for $\eta=1$.
Then there exists a constant $C>0$ such that
\[
\sup_{t\geq0}\E[|X^{x}(t)|]\leq C(1+|x|),\qquad x\in\R_{+}^{m}.
\]
\end{Proposition}
\begin{proof}
Let $\varkappa>0$ be such that
$|\mathrm{e}^{\beta t}y|\leq e^{-\varkappa t}|y|$ for all $y\in\R^{m}$
and set
\[
\widetilde{b}=b+\int_{\R_{+}^{m}}z\nu(\mathrm{d}z).
\]
Using first the sub-additivity of the square-root, then the Cauchy-Schwartz
inequality and finally \eqref{eq:first moment} we find that
\begin{align*}
\E[|X^{x}(t)|] & \leq\sum_{k=1}^{m}\E[X_{k}^{x}(t)]\\
 & \leq\sqrt{m}|\E[X^{x}(t)]|\\
 & \leq\sqrt{m}|\mathrm{e}^{\beta t}x|+\sqrt{m}\int_{0}^{t}|\mathrm{e}^{\beta s}\widetilde{b}|\mathrm{d}s\\
 & \leq\sqrt{m}\mathrm{e}^{-\varkappa t}|x|+\sqrt{m}\int_{0}^{t}\mathrm{e}^{-\varkappa s}|\widetilde{b}|\mathrm{d}s\\
 & \leq\sqrt{m}|x|+\sqrt{m}\frac{|\widetilde{b}|}{\varkappa},
\end{align*}
which proves the assertion. \end{proof}

\section{Lyapunov estimate for the extended generator}
 Recall that $|x|_M$ is defined by \eqref{eq:02}, $V(x) = (1 + |x|_M^2)^{1/2}$
 and observe that we can find constants $c^* \geq c_* > 0$ such that
 \begin{align}\label{equivalence norms}
  c_* |x| \leq |x|_M \leq c^* |x|, \qquad x \in \R_+^m.
 \end{align}
 Let $L_0$ be the extended generator of the anisotropic stable JCIR process $(Y^x(t))_{t \geq 0}$ obtained from \eqref{eq: a roots model}
 whose subordinator $\nu$ has only small jumps, i.e., $\nu( \{ |z| > 1\} ) = 0$.
\begin{Lemma}\label{Lyapunov estimate}
Suppose that $\beta $ has only eigenvalues with negative real-parts.
Then $V$ belongs to the domain of the extended generator $L_0$, one has
 \begin{align*}
     L_0V(x) &= \langle b + \beta x, \nabla V(x) \rangle + \int_{ \{ |z| \leq 1\} } \left( V(x+z) - V(x) \right) \nu(\mathrm{d}z)
     \\ &\ \ \ + \sum_{j=1}^{m}\sigma_j x_j \int_0^{\infty}\left( V(x + e_j z) - V(x) - z \frac{\partial V(x)}{\partial x_j} \right)\mu_{\alpha_j}(\mathrm{d}z), \qquad x \in \R_+^m
 \end{align*}
 and there exists two constants $c_1, c_2 > 0$ such that
 \begin{align}\label{eq:03}
  L_0V(x) \leq - c_1 V(x) + c_2, \qquad x \in \R_+^m.
 \end{align}
\end{Lemma}
\begin{proof}

 By direct computation one finds that
 \begin{align*}
  \nabla V(x) = \frac{Mx}{V(x)} \ \text{ and } \ \frac{\partial^2 V(x)}{\partial x_j \partial x_k} = \frac{M_{jk}}{V(x)} - \frac{(Mx)_k (Mx)_j}{V(x)^3},
 \end{align*}
 which, together with \eqref{equivalence norms}, imply $|\nabla V(x)| \leq C$ and
 $\left|\frac{\partial^2 V(x)}{\partial x_j \partial x_k}\right| \leq
 CV(x)^{-1}$ for all $k,j \in \{1,\dots, m\}$. Here and below we use $C$ to denote a generic positive constant whose precise value is not important and may vary from time to time.
 By the mean-value theorem we obtain
 \begin{align*}
     |V(x + z) - V(x)| \leq C |z|, \qquad x,z \in \R_+^m,
 \end{align*}
 and applying the mean-value theorem twice gives
 \begin{align*}
     \left| V(x+ e_j z) - V(x) - z\frac{\partial V(x)}{\partial x_j}\right|
      &=  \left| z^2 \int_0^1 \int_0^1 \frac{\partial^2 V(x + e_j z s r)}{\partial x_j^2} \mathrm{d}r \mathrm{d}s \right|
     \\ &\leq  Cz^2 \int_0^1 \int_0^1 \frac{1}{V(x+e_j z rs)} \mathrm{d}r \mathrm{d}s
     \\ &\leq C \frac{z^2}{V(x)},
 \end{align*}
 where we have used $V(x+e_j z rs) \geq ( 1 + c_*^2|x + e_j z rs|^2)^{1/2}
 \geq (1 + c_*^2|x|^2)^{1/2} \geq C V(x)$.
 Hence all integrals in $L_0V$ are well defined and one easily finds that
 $|L_0V(x)| \leq CV(x)$, $x \in \R_+^m$.
 Applying the It\'o formula gives $V(Y_t^x) = V(x) + \int_0^tL_0V(Y_s^x)\mathrm{d}s + M_t(V)$, where $(M_t(V))_{t \geq 0}$ is a local martingale.
 Using the fact that $Y_t^x$ has finite first moment combined with the particular form of $M_t(V)$, one can easily show that $(M_t(V))_{t \geq 0}$ is, indeed, a true martingale. Hence taking expectations gives $\E[ V(Y_t^x) ] = V(x) + \int_0^t \E[L_0V(Y_s^x)] \mathrm{d}s$, i.e., $V$ belongs to the domain of the extended generator $L_0$ and has the desired form.

 It remains to prove \eqref{eq:03}. By continuity,
 for $|x|_M \leq 1$ one clearly has
 $L_0V(x) \leq |L_0V(x)| \leq C$. Take $x \in \R_+^m$ with $|x|_M > 1$.
 For the drift we obtain
  \begin{align*}
     \langle b, \nabla V(x) \rangle &\le |\langle b, \nabla V(x) \rangle |\le C.
 \end{align*}
 Likewise, using the identity $M\beta + \beta^{\top}M = -\mathbbm{1}$ we find that
 \begin{align*}
 \langle \beta x, \nabla V(x) \rangle &= \frac{1}{2} \langle M\beta x + \beta^{\top}Mx, x \rangle V(x)^{-1}
 \\ &\leq - \frac{1}{2}|x|^2 V(x)^{-1}
 \\ &\leq - \frac{1}{2(c^*)^{2}}|x|_M^2 V(x)^{-1}
 \\ &\leq - \frac{c_*^2}{2\sqrt{2}(c^*)^{2}}|x|_M
 \\ &\leq - \frac{c_*^2}{4(c^*)^{2}}V(x),
 \end{align*}
 where we have used \eqref{equivalence norms}
 and $V(x) \leq \sqrt{2}|x|_M$ since $|x|_M > 1$. For the state-independent jumps we obtain
 \[
  \int_{\{ |z| \leq 1\} } \left( V(x+z) - V(x) \right) \nu(\mathrm{d}z)
  \leq C,
 \]
 while the state-dependent jumps can be estimated by
 \begin{align*}
     &\ \sum_{j=1}^{m}\sigma_j x_j \int_{0}^{\infty}\left( V(x + e_j z) - V(x) - z \frac{\partial V(x)}{\partial x_j} \right) \mu_{\alpha_j}(\mathrm{d}z)
     \\ &= \sum_{j=1}^{m}\sigma_j x_j \int_{0}^{R}\left( V(x + e_j z) - V(x) - z \frac{\partial V(x)}{\partial x_j} \right) \mu_{\alpha_j}(\mathrm{d}z)
     \\ & \ \ \ + \sum_{j=1}^{m}\sigma_j x_j \int_{R}^{\infty}\left( V(x + e_j z) - V(x) - z \frac{\partial V(x)}{\partial x_j} \right) \mu_{\alpha_j}(\mathrm{d}z)
     \\ &\leq C \sum_{j=1}^{m}\frac{x_j}{V(x)} \int_{0}^{R} z^2 \mu_{\alpha_j}(\mathrm{d}z)
      + C \sum_{j=1}^{m}x_j \int_{R}^{\infty}z \mu_{\alpha_j}(\mathrm{d}z)
     \\ &\leq C \max_{j \in \{1,\dots,m\}} \int_{0}^{R}z^2 \mu_{\alpha_j}(\mathrm{d}z) + C \max_{j \in \{1,\dots, m\}}\int_R^{\infty}z \mu_{\alpha_j}(\mathrm{d}z) V(x),
 \end{align*}
 where $R > 0$ is some constant to be fixed below.
 Combining all estimates we obtain
 \[
  L_0V(x) \leq C\left( 1 + \max_{j \in \{1,\dots,m\}} \int_{0}^{R}z^2 \mu_{\alpha_j}(\mathrm{d}z)\right) - \left( \frac{c_*^2}{4(c^*)^{2}} - C\max_{j \in \{1,\dots, m\}}\int_R^{\infty}z \mu_{\alpha_j}(\mathrm{d}z)\right)V(x).
 \]
 Choosing $R$ large enough such that
 \[
 C\max_{j \in \{1,\dots, m\} } \int_R^{\infty}z \mu_{\alpha_j}(\mathrm{d}z) \le \frac{c_*^2}{8(c^*)^{2}},
 \]
 the assertion is proved.
\end{proof}

\section{Proof of Theorem \ref{th:04}}

Let $\lambda>0$ and $(a_{1},\dots,a_{m})$ be the anisotropy defined
in \eqref{MAIN:04}. The anisotropic H\"older-Zygmund space $C_{b}^{\lambda,a}(\R^{m})$
is defined as the Banach space of functions $\phi$ with finite norm
\[
\|\phi\|_{C_{b}^{\lambda,a}}=\|\phi\|_{\infty}+\sum\limits _{k=1}^{m}\sup\limits _{h\in[-1,1]}|h|^{-\lambda/a_{k}}\|\Delta_{he_{k}}\phi\|_{\infty}.
\]
The following lemma provides our main technical tool for the proof
of Theorem \ref{th:04}.
\begin{Lemma}\label{LEMMA:02}
Let $\lambda,\eta>0$
be such that $(\lambda+\eta)/a_{k}\in(0,1)$ for all $k=1,\dots,d$.
Suppose that $G$ is a finite measure over $\R^{m}$ and there exists
$A>0$ such that for all $\phi\in C_{b}^{\eta,a}(\R^{m})$, all $k=1,\dots,m$
and all $h\in[-1,1]$
\begin{align}
\left|\int_{\R^{m}}(\phi(x+he_{k})-\phi(x))G(\mathrm{d}x)\right|\leq A\|\phi\|_{C_{b}^{\eta,a}}|h|^{(\lambda+\eta)/a_{k}}.\label{EQ:32}
\end{align}
Then there exists $g\in B_{1,\infty}^{\lambda,a}(\R^{m})$ such that
$G(\mathrm{d}x)=g(x)\mathrm{d}x$ and
\begin{align}
\|g\|_{B_{1,\infty}^{\lambda,a}}\leq G(\R^{m})+3mA(2m)^{\eta/\lambda}\left(1+\frac{\lambda}{\eta}\right)^{1+\frac{\eta}{\lambda}}.\label{MAIN:05}
\end{align}
\end{Lemma}
This lemma was first proved in \cite[Lemma 2.1]{DF13}
and \cite{DR14} for the isotropic case $a_{1}=\dots=a_{m}$. Above
anisotropic version is due to \cite{FJR18a}. The essential step in
the process of proving Theorem \ref{th:04} is based on a suitable
application of this lemma to the finite measure $G_{t}(\mathrm{d}y)=\rho_{\delta}(y)P_{t}(x,\mathrm{d}y)$.
In order to prove \eqref{EQ:32} for $G_{t}$, we use an approximation
similar to \cite{DF13,FJR18,FJR18a}. From now on we fix $x\in\R_{+}^{m}$
and let $X^{x}=(X^{x}(t))_{t\geq0}$ be the unique solution to \eqref{eq: a roots model}
with $X^{x}(0)=x$ and $\nu$ satisfying
\[
\int_{\R_{+}^{m}}\mathbbm{1}_{\{|z|>1\}}|z|^{1+\tau}\nu(\mathrm{d}z)<\infty
\]
for some $\tau>0$. To simplify the notation we let $(X(t))_{t\geq0}$
stand for $(X^{x}(t))_{t\geq0}$.

\subsection{Short time approximation}

For $\varepsilon\in(0,1\wedge t)$, define the approximation $X^{\e}(t)=(X_{1}^{\e}(t),\dots,X_{m}^{\e}(t))$
by
\begin{align}
X_{i}^{\e}(t) & =X_{i}(t-\e)+\left(b_{i}+\sum_{k=1}^{m}\beta_{ik}X_{k}(t-\e)\right)\e\label{eq:07}\\
 & \ \ \ +\sigma_{i}^{1/\alpha_{i}}X_{i}(t-\e)^{1/\alpha_{i}}(Z_{i}(t)-Z_{i}(t-\e))+(J_{i}(t)-J_{i}(t-\e)),\nonumber
\end{align}
where $i=1,\dots,m$. Define $\kappa_{1},\dots,\kappa_{m}>0$ by
\begin{align}
\kappa_{i}=\min\left\{ 1+\frac{1}{\alpha_{\max}},\frac{1}{\alpha_{i}}+\frac{1}{\alpha_{i}^{2}}\right\} ,\qquad i=1,\dots,m.\label{eq:14}
\end{align}
The next proposition shows that the convergence rate for $X_{i}^{\e}(t)\to X_{i}(t)$
as $\e\to0$ is precisely given by $\kappa_{i}$.
\begin{Proposition}\label{PROP:00}
Let $i\in\{1,\dots,m\}$ be arbitrary. The following assertions hold:
\begin{enumerate}
\item[(a)] For each $\eta\in(0,\left(1+\tau\right)\wedge\alpha_{\min})$
and $T\ge1$, there exists a constant $C=C(\eta,T)>0$
such that, for all $0\leq s\leq t\leq s+1 \le T$, it
holds that
\[
\E[|X_{i}(t)-X_{i}(s)|^{\eta}]\leq C(1+|x|)^{\eta}(t-s)^{\eta/\alpha_{i}}.
\]
\item[(b)] For each $\eta\in(0,1)$ and $T>0$, there exists
a constant $C=C(\eta,T)>0$ such that
\[
\E[|X_{i}(t)-X_{i}^{\e}(t)|^{\eta}]\leq C(1+|x|)^{\eta}\e^{\eta\kappa_{i}},\ \ t\in(0,T],\ \ \e\in(0,1\wedge t).
\]
\end{enumerate}
\end{Proposition}
\begin{proof}
Fix constants $\gamma_{1},\dots,\gamma_{m}$,
satisfying for each $i=1,\dots,m$
\begin{align*}
\gamma_{i}\in(\alpha_{i},2)\ \text{ and }\ \frac{\gamma_{i}}{\alpha_{i}}<\min\{1+\tau,\alpha_{\min}\}.
\end{align*}
In the following we will use $C$ to denote a positive
constant, whose exact value is not important and may change from time
to time.

\textit{(a)} Write $\E[|X_{i}(t)-X_{i}(s)|^{\eta}]\leq R_{1}+R_{2}+R_{3}$,
where
\begin{align*}
R_{1} & =C\E\left[\left|\int_{s}^{t}\left(b_{i}+\sum_{k=1}^{m}\beta_{ik}X_{k}(u)\right)\mathrm{d}u\right|^{\eta}\right],\\
R_{2} & =C\E\left[\left|\int_{s}^{t}X_{i}(u-)^{1/\alpha_{i}}\mathrm{d}Z_{i}(u)\right|^{\eta}\right],\\
R_{3} & =C\E\left[|J_{i}(t)-J_{i}(s)|^{\eta}\right].
\end{align*}
If $\eta\in(0,1]$, then we use the Jensen inequality and Proposition
\ref{moment estimate} to obtain
\begin{align*}
R_{1} & \leq C\left(\E\left[\left|\int_{s}^{t}\left(b_{i}+\sum_{k=1}^{m}\beta_{ik}X_{k}(u)\right)\mathrm{d}u\right|\right]\right)^{\eta}\\
 & \leq C(t-s)^{\eta}+C(t-s)^{\eta}\sup_{u\in[s,t]}\left(\E[|X(u)|]\right)^{\eta}\\
 & \leq C(1+|x|)^{\eta}(t-s)^{\eta}.
\end{align*}
If $\eta\in(1,\alpha_{\min})$, then we use the H\"older inequality
with $\frac{1}{\eta}+\frac{1}{\frac{\eta}{\eta-1}}=1$ to obtain
\begin{align*}
R_{1} & \leq C(t-s)^{\eta-1}\int_{s}^{t}\E\left[\left|b_{i}+\sum_{k=1}^{m}\beta_{ik}X_{k}(u)\right|^{\eta}\right]\mathrm{d}u\\
 & \leq C(t-s)^{\eta-1}(t-s)\left(1+\sup\limits _{u\in[s,t]}\E[|X(u)|^{\eta}]\right)\\
 & \leq C\left(1+|x|\right)^{\eta}(t-s)^{\eta}.
\end{align*}
Combining both cases $\eta\in(0,1]$ and $\eta\in(1,\alpha_{\min})$
we find that $R_{1}\leq C\left(1+|x|\right)^{\eta}(t-s)^{\eta}$.
For the second term we apply Lemma \ref{LEMMA:00} to obtain
\begin{align}
R_{2}\leq C(t-s)^{\eta/\alpha_{i}}\sup\limits _{u\in[s,t]}\left(\E[X_{i}(u)^{\gamma_{i}/\alpha_{i}}]\right)^{\eta/\gamma_{i}}.\label{eq:06}
\end{align}
Since $\gamma_{i}$ also satisfies $\gamma_{i}/\alpha_{i}<\min\{1+\tau,\alpha_{\min}\}$,
we may apply Proposition \ref{moment estimate} to find that
\begin{align*}
\sup\limits _{u\in[s,t]}\E[X_{i}(u)^{\gamma_{i}/\alpha_{i}}]\leq\sup\limits _{u\in[s,t]}\E[|X(u)|^{\gamma_{i}/\alpha_{i}}]\leq C\left(1+|x|\right)^{\gamma_{i}/\alpha_{i}}.
\end{align*}
Inserting this into \eqref{eq:06} gives $R_{2}\leq C(1+|x|)^{\eta/\alpha_{i}}(t-s)^{\eta/\alpha_{i}}$.
For the last term we may apply the estimates for the stochastic integrals
from \cite[appendix]{FJR18} to find that $R_{3}=C\E[J_{i}(t-s)^{\eta}]\leq C(t-s)^{\eta\wedge1}$.
Combining all estimates for $R_{1},R_{2},R_{3}$ yields
\begin{align*}
\E[|X_{i}(t)-X_{i}(s)|^{\eta}] & \leq C(1+|x|)^{\eta}(t-s)^{\eta\wedge1}+C(1+|x|)^{\eta/\alpha_{i}}(t-s)^{\eta/\alpha_{i}}\\
 & \leq C(1+|x|)^{\eta}(t-s)^{\eta/\alpha_{i}}.
\end{align*}
This proves the assertion.

\textit{(b)} Write $\E[|X_{i}(t)-X_{i}^{\e}(t)|^{\eta}]\leq R_{1}+R_{2}$,
where
\begin{align*}
R_{1} & =\E\left[\left|\int_{t-\e}^{t}\left(\sum_{k=1}^{m}\beta_{ik}(X_{k}(u)-X_{k}(t-\e))\right)\mathrm{d}u\right|^{\eta}\right],\\
R_{2} & =\sigma_{i}^{\eta/\alpha_{i}}\E\left[\left|\int_{t-\e}^{t}(X_{i}(u-)^{1/\alpha_{i}}-X_{i}(t-\e)^{1/\alpha_{i}})\mathrm{d}Z_{i}(u)\right|^{\eta}\right].
\end{align*}
For the first term we use part (a) and the fact that $\eta\in(0,1)$
to obtain
\begin{align*}
R_{1} & \leq\left(\E\left[\int_{t-\e}^{t}\sum_{k=1}^{m}|\beta_{ik}||X_{k}(u)-X_{k}(t-\e)|\mathrm{d}u\right]\right)^{\eta}\\
 & \leq C\e^{\eta}\sum_{k=1}^{m}\sup\limits _{u\in[t-\e,t]}\left(\E\left[ |X_{k}(u)-X_{k}(t-\e)| \right]\right)^{\eta}\\
 & \leq C\e^{\eta}\sum_{k=1}^{m}(1+|x|)^{\eta}\e^{\eta/\alpha_{k}}\\
 & \leq C(1+|x|)^{\eta}\e^{\eta+\eta/\alpha_{\max}}.
\end{align*}
Let us turn to the second term. We use Lemma \ref{LEMMA:00} and then
$|y^{1/\alpha_{i}}-z^{1/\alpha_{i}}|\leq|y-z|^{1/\alpha_{i}}$
for $y,z\ge0$ to find that
\begin{align*}
R_{2} & \leq C\e^{\eta/\alpha_{i}}\sup\limits _{u\in[t-\e,t]}\left(\E\left[|X_{i}(u)^{1/\alpha_{i}}-X_{i}(t-\e)^{1/\alpha_{i}}|^{\gamma_{i}}\right]\right)^{\eta/\gamma_{i}}\\
 & \leq C\e^{\eta/\alpha_{i}}\sup\limits _{u\in[t-\e,t]}\left(\E[|X_{i}(u)-X_{i}(t-\e)|^{\gamma_{i}/\alpha_{i}}]\right)^{\eta/\gamma_{i}}.
\end{align*}
Since $\gamma_{i}/\alpha_{i}<\min\{1+\tau,\alpha_{\min}\}$, we may
apply part (a) which gives
\begin{align*}
\sup\limits _{u\in[t-\e,t]}\E[|X_{i}(u)-X_{i}(t-\e)|^{\gamma_{i}/\alpha_{i}}]\leq C(1+|x|)^{\gamma_{i}/\alpha_{i}}\e^{\gamma_{i}/\alpha_{i}^{2}}
\end{align*}
and hence
\[
R_{2}\leq C(1+|x|)^{\eta/\alpha_{i}}\e^{\frac{\eta}{\alpha_{i}}\left(1+1/\alpha_{i}\right)}.
\]
This proves the assertion. \end{proof}

\subsection{The key estimate}

Recall that $\kappa_{1},\dots,\kappa_{m}$ are defined in \eqref{eq:14}.
Based on the previous approximation we show the following.
\begin{Proposition}\label{PROP:01}
Let $t>0$ be arbitrary and fixed. Take $\varkappa\in(0,1/\alpha_{\max}]$
and let $\eta\in(0,\varkappa a_{\min})$. Then there
exists a constant $C>0$ such that, for any $\e\in(0,1\wedge t)$,
$h\in[-1,1]$, $\phi\in C_{b}^{\eta,a}(\R^{m})$ and $i\in\{1,\dots,m\}$,
\begin{align*}
 & \ \left|\E\left[\rho_{\delta}(X(t))\Delta_{he_{i}}\phi(X(t))\right]\right|\\
 & \leq C\|\phi\|_{C_{b}^{\eta,a}}(1+|x|)^{\varkappa}\left(|h|^{\eta/a_{i}}\e^{\varkappa/\alpha_{\max}}+|h|\e^{-1/\alpha_{i}}+\max\limits _{j\in\{1,\dots,m\}}\e^{\eta\kappa_{j}/a_{j}}\right).
\end{align*}
\end{Proposition} \begin{proof} For $\e\in(0,1\wedge t)$ let $X^{\e}(t)=(X_{1}^{\e}(t),\dots,X_{m}^{\e}(t))$
be given as in \eqref{eq:07}. Then
\begin{align*}
\left|\E\left[\rho_{\delta}(X(t))\Delta_{he_{i}}\phi(X(t))\right]\right|\leq R_{1}+R_{2}+R_{3},
\end{align*}
where $R_{1},R_{2},R_{3}$ are given by
\begin{align*}
R_{1} & =\left|\E\left[\Delta_{he_{i}}\phi(X(t))\left(\rho_{\delta}(X(t))-\rho_{\delta}(X(t-\e))\right)\right]\right|,\\
R_{2} & =\E\left[|\Delta_{he_{i}}\phi(X(t))-\Delta_{he_{i}}\phi(X^{\e}(t))|\rho_{\delta}(X(t-\e)\right],\\
R_{3} & =\left|\E\left[\rho_{\delta}(X(t-\e))\Delta_{he_{i}}\phi(X^{\e}(t))\right]\right|.
\end{align*}
For the first term we use $\rho_{\delta}\leq1$ and $\varkappa\le1/\alpha_{k}$
to find
\begin{align*}
|\rho_{\delta}(x)-\rho_{\delta}(y)| & \leq2\wedge\left(\sum_{k=1}^{m}|x_{k}^{1/\alpha_{k}}-y_{k}^{1/\alpha_{k}}|\right)\\
 & \leq\sum_{k=1}^{m}2\wedge|x_{k}-y_{k}|^{1/\alpha_{k}}\\
 & \leq C\sum_{k=1}^{m}|x_{k}-y_{k}|^{\varkappa},
\end{align*}
and hence deduce from Proposition \ref{PROP:00}.(a) that
\begin{align*}
R_{1} & \leq C\|\phi\|_{C_{b}^{\eta,a}}|h|^{\eta/a_{i}}\sum_{k=1}^{m}\E[|X_{k}(t)-X_{k}(t-\e)|^{\varkappa}]\\
 & \leq C\|\phi\|_{C_{b}^{\eta,a}}|h|^{\eta/a_{i}}\sum_{k=1}^{m}(1+|x|)^{\varkappa}\e^{\varkappa/\alpha_{k}}\\
 & \leq C\|\phi\|_{C_{b}^{\eta,a}}|h|^{\eta/a_{i}}(1+|x|)^{\varkappa}\e^{\varkappa/\alpha_{\max}}.
\end{align*}
For $R_{2}$ we first use that $\phi\in C_{b}^{\eta,a}(\R^{m})$,
then $\rho_{\delta}\leq1$ and finally Proposition \ref{PROP:00}.(b)
to obtain
\begin{align*}
R_{2} & \leq C\|\phi\|_{C_{b}^{\eta,a}}\max\limits _{j\in\{1,\dots,m\}}\E\left[|X_{j}(t)-X_{j}^{\e}(t)|^{\eta/a_{j}}\right]\\
 & \leq C\|\phi\|_{C_{b}^{\eta,a}}\max\limits _{j\in\{1,\dots,m\}}\e^{\eta\kappa_{j}/a_{j}}(1+|x|)^{\eta/a_{j}}\\
 & \leq C\|\phi\|_{C_{b}^{\eta,a}}(1+|x|)^{\varkappa}\max\limits _{j\in\{1,\dots,m\}}\e^{\eta\kappa_{j}/a_{j}}.
\end{align*}
Let us turn to $R_{3}$. Define $\sigma(x)=\mathrm{diag}((\sigma_{1}x_{1})^{1/\alpha_{1}},\dots,(\sigma_{m}x_{m})^{1/\alpha_{m}})$.
Let $f_{t}^{Z}$ be the density of $Z(t)=(Z_{1}(t),\dots,Z_{m}(t))$.
Using \eqref{eq:07} we find that
\[
X^{\e}(t)=U^{\e}(t)+\sigma(X(t-\e))(Z(t)-Z(t-\e)),
\]
with $U^{\e}(t)=(U_{1}^{\e}(t),\dots,U_{m}^{\e}(t))$ being given
by
\[
U_{i}^{\e}(t)=X_{i}(t-\e)+\left(b_{i}+\sum_{k=1}^{m}\beta_{ik}X_{k}(t-\e)\right)\e+(J_{i}(t)-J_{i}(t-\e)).
\]
Finally note that $\sigma(X(t-\e))^{-1}=\mathrm{diag}((\sigma_{1}X_{1}(t-\e))^{-1/\alpha_{1}},\dots,(\sigma_{m}X_{m}(t-\e))^{-1/\alpha_{m}})$
is well-defined, since $\P[X(t-\e)\in\R_{++}^{m}]=1$ holds by Proposition
\ref{th:boundary behavior}. Then we obtain for each $\e\in(0,1\wedge t)$,
\begin{align*}
R_{3} & =\left|\E\left[\int_{\R^{m}}\rho_{\delta}(X(t-\e))(\Delta_{he_{i}}\phi)(U^{\e}(t)+\sigma(X(t-\e))z)f_{\e}^{Z}(z)\mathrm{d}z\right]\right|\\
 & =\left|\E\left[\int_{\R^{m}}\rho_{\delta}(X(t-\e))\phi(U^{\e}(t)+\sigma(X(t-\e))z)(\Delta_{-h\sigma(X(t-\e))^{-1}e_{i}}f_{\e}^{Z})(z)\mathrm{d}z\right]\right|\\
 & \leq\|\phi\|_{\infty}\E\left[\rho_{\delta}(X(t-\e))\int_{\R^{m}}|(\Delta_{-h\sigma(X(t-\e))^{-1}e_{i}}f_{\e}^{Z})(z)|\mathrm{d}z\right]\\
 & \leq\|\phi\|_{\infty}|h|\sigma_{i}^{-1/\alpha_{i}}\E\left[\rho_{\delta}(X(t-\e))X_{i}(t-\e)^{-1/\alpha_{i}}\int_{\R^{m}}\left|\frac{\partial f_{\e}^{Z}(z)}{\partial z_{i}}\right|\mathrm{d}z\right]\\
 & \leq C\|\phi\|_{C_{b}^{\eta,a}}|h|\e^{-1/\alpha_{i}},
\end{align*}
where we have used Lemma \ref{LEMMA:01} and $\rho_{\delta}(x)x_{i}^{-1/\alpha_{i}}\leq1$.
Summing up the estimates for $R_{1},R_{2},R_{3}$ yields the assertion.
\end{proof}

\subsection{Concluding the proof of Theorem \ref{th:04}}

Below we provide the proof of Theorem \ref{th:04}. Fix $t>0$ and
$x\in\R_{+}^{m}$. We will show that Lemma \ref{LEMMA:02} applies
to the finite measure $G_{t}(x,\mathrm{d}y)=\rho_{\delta}(y)P_{t}(x,\mathrm{d}y)$.
Using the particular form of $\kappa_{j}$ we obtain $\kappa_{j}\alpha_{j}>1$
and hence $\kappa_{j}/a_{j}>1/\overline{\alpha}$ for all $j\in\{1,\dots,m\}$.
This implies
\[
\frac{a_{j}}{\kappa_{j}}\frac{1}{a_{i}}<\frac{\overline{\alpha}}{a_{i}}=\alpha_{i},\ \ i,j\in\{1,\dots,m\}.
\]
Hence we find $\eta\in(0,1)$ and $c_{1},\dots,c_{m}>0$ such that,
for all $i,j\in\{1,\dots,m\}$,
\[
0<\eta<a_{i}\varkappa,\qquad\frac{a_{j}}{\kappa_{j}}\frac{1}{a_{i}}<c_{i}<\alpha_{i}\left(1-\frac{\eta}{a_{i}}\right).
\]
Define
\[
\lambda=\min_{i,j\in\{1,\dots,m\}}\left\{ \varkappa c_{i}a_{i}/\alpha_{\max},\ a_{i}-\eta-\frac{a_{i}c_{i}}{\alpha_{i}},\ \eta\left(c_{i}a_{i}\frac{\kappa_{j}}{a_{j}}-1\right)\right\} >0.
\]
Let $\phi\in C_{b}^{\eta,a}(\R^{m})$. By Proposition \ref{PROP:01}
we obtain, for $h\in[-1,1]$, $\e=|h|^{c_{i}}(1\wedge t)$ and $i\in\{1,\dots,m\}$,
\begin{align*}
 & \ \left|\E\left[\rho_{\delta}(X(t))\Delta_{he_{i}}\phi(X(t))\right]\right|\\
 & \leq C\|\phi\|_{C_{b}^{\eta,a}}(1+|x|)^{\varkappa}\left(|h|^{\eta/a_{i}}\e^{\varkappa/\alpha_{\max}}+|h|\e^{-1/\alpha_{i}}+\max\limits _{j\in\{1,\dots,m\}}\e^{\eta\kappa_{j}/a_{j}}\right)\\
 & \leq\frac{C\|\phi\|_{C_{b}^{\eta,a}}}{(1\wedge t)^{1/\alpha_{i}}}(1+|x|)^{\varkappa}\left(|h|^{\eta/a_{i}+c_{i}\varkappa/\alpha_{\max}}+|h|^{1-c_{i}/\alpha_{i}}+\max\limits _{j\in\{1,\dots,m\}}|h|^{c_{i}\eta\kappa_{j}/a_{j}}\right)\\
 & =\frac{C\|\phi\|_{C_{b}^{\eta,a}}}{(1\wedge t)^{1/\alpha_{i}}}|h|^{\eta/a_{i}}(1+|x|)^{\varkappa}\left(|h|^{c_{i}\varkappa/\alpha_{\max}}+|h|^{1-\eta/a_{i}-c_{i}/\alpha_{i}}+\max\limits _{j\in\{1,\dots,m\}}|h|^{c_{i}\eta\kappa_{j}/a_{j}-\eta/a_{i}}\right)\\
 & \leq\frac{C\|\phi\|_{C_{b}^{\eta,a}}}{(1\wedge t)^{1/\alpha_{i}}}(1+|x|)^{\varkappa}|h|^{(\eta+\lambda)/a_{i}}.
\end{align*}
This shows that Lemma \ref{LEMMA:02} is applicable to $G_{t}(x,\mathrm{d}y)$.
Hence $G_{t}(x,\mathrm{d}y)$ has a density $g_{t}(x,y)$, i.e., $\rho_{\delta}(y)P_{t}(x,\mathrm{d}y)=g_{t}(x,y)\mathrm{d}y$.
In view of \eqref{MAIN:05} this density satisfies
\begin{align*}
\|g_{t}(x,\cdot)\|_{B_{1,\infty}^{\lambda,a}(\R_{+}^{m})} & \leq\int_{\R_{+}^{m}}\rho_{\delta}(y)P_{t}(x,\mathrm{d}y)+C(t)(1+|x|)^{\varkappa}(1\wedge t)^{-1/\alpha_{\min}}\\
 & \leq C(t)(1+|x|)^{\varkappa}(1\wedge t)^{-1/\alpha_{\min}},
\end{align*}
where we have used $\rho\leq1$ and $C(t)$ is a generic constant
which is locally bounded in $t\geq0$. Since $\rho_{\delta}(y)>0$
for $y\in\R_{++}^{m}$, $P_{t}(x,\mathrm{d}y)$ has also a density
$p_{t}(x,y)$ on $\R_{++}^{m}$ which gives $P_{t}(x,\mathrm{d}y)=p_{t}(x,y)\mathrm{d}y+P_{t}^{\mathrm{sing}}(x,\mathrm{d}y)$,
where $P_{t}^{\mathrm{sing}}(x,\mathrm{d}y)$ is supported on $\partial\R_{+}^{m}$.
Using Proposition \ref{th:boundary behavior} we conclude that $P_{t}^{\mathrm{sing}}(x,\mathrm{d}y)=0$
and hence $p_{t}^{\delta}(x,y)=g_{t}(x,y)$. This proves the assertion
of Theorem \ref{th:04}.

\section{Some results on ergodicity in total variation norm}

In this section we briefly summarize some results on geometric ergodicity
in the total variation distance for continuous-time Markov processes.
For additional details we refer to \cite{H16} and \cite{Kulik18}.
Let $E$ be a Polish space and let $(X_{t})_{t\geq0}$ be a Feller
process on $E$. Denote by $(P_{t}(x,\mathrm{d}y))_{t\geq0}$ its
transition probabilities and by $L$ the extended generator. \begin{Theorem}\label{th:3}
Suppose that the following conditions are satisfied:
\begin{enumerate}
\item[(a)] There exists a continuous function $V:E\longmapsto[1,\infty)$ which
belongs to the domain of the extended generator such that
\[
LV(x)\leq-aV(x)+M,\qquad x\in E,
\]
where $a,M>0$ are some constants. Moreover, for each $R>0$ the level
sets $\{(x,y)\in E^{2}\ |\ V(x)+V(y)\leq R\}$ are compact.
\item[(b)] For each $R>0$, there exists $h>0$ and $\delta\in(0,2)$ such that
\[
\|P_{h}(x,\cdot)-P_{h}(y,\cdot)\|_{\mathrm{TV}}\leq2-\delta
\]
holds for all $x,y\in E$ with $V(x)+V(y)\leq R$.
\end{enumerate}
Then there exists constants $C,\beta>0$ such that for all $t\geq0$
and $x,y\in E$,
\[
\|P_{t}(x,\cdot)-P_{t}(y,\cdot)\|_{\mathrm{TV}}\leq Ce^{-\beta t}\left(V(x)+V(y)\right).
\]
Moreover, there exists a unique invariant probability measure $\pi$.
This measure satisfies
\begin{align}
\int_{E}V(x)\pi(\mathrm{d}x)<\infty\label{eq:05}
\end{align}
and for all $t\geq0$ and $x\in E$ one has
\[
\|P_{t}(x,\cdot)-\pi\|_{\mathrm{TV}}\leq Ce^{-\beta t}\left(V(x)+\int_{E}V(y)\pi(\mathrm{d}y)\right).
\]
\end{Theorem} A proof of this Theorem is given in \cite[Theorem 4.1]{H16}.
Moreover, the same result can also be obtained from a combination
of Corollary 2.8.3 and Theorem 3.2.3 in \cite{Kulik18}, where also
additional comments and examples are given.


\section{Some properties of the cylindrical L\'evy process $(Z_{1},\dots,Z_{m})$}

Observe that the L\'evy process $Z=(Z_{1},\dots,Z_{m})$ has symbol
\[
\Psi_{Z}(\xi)=\int_{\R_{+}^{m}}\left(\mathrm{e}^{i\langle\xi,z\rangle}-1-\mathrm{i}\langle\xi,z\rangle\right)\mu(\mathrm{d}z)=\Psi_{\alpha_{1}}(\xi_{1})+\dots+\Psi_{\alpha_{m}}(\xi_{m}),
\]
where the L\'evy measure $\mu$ is given by
\[
\mu(\mathrm{d}z)=\sum_{k=1}^{m}\mu_{\alpha_{k}}(\mathrm{d}z_{k})\otimes\prod_{j\neq k}\delta_{0}(\mathrm{d}z_{j}).
\]
The next lemma is standard and follows from the scaling property $Z_{j}(t)=t^{1/\alpha_{j}}Z_{j}(1)$,
$j=1,\dots,m$, where equality holds in the sense of distributions.
\begin{Lemma}\label{LEMMA:01} $Z(t)$ has for each $t>0$ a smooth
density $f_{t}^{Z}$ on $\R^{m}$. Moreover, there exists a constant
$C>0$ such that
\[
\int_{\R^{m}}\left|\frac{\partial f_{t}^{Z}(z)}{\partial z_{j}}\right|\mathrm{d}z\leq Ct^{-1/\alpha_{j}},\qquad t>0.
\]
\end{Lemma} Below we state some useful estimates on stochastic integrals
with respect to the L\'evy processes $Z_{1},\dots,Z_{m}$ due to \cite[Lemma A.2]{DF13}.
\begin{Lemma}\label{LEMMA:00} Let $0<\eta\leq\alpha_{j}<\gamma\leq2$.
Then there exists a constant $C=C(\eta,\gamma)>0$ such that, for
any predictable process $H(u)$ and $0\leq s\leq t\leq s+1$,
\[
\E\left[\left|\int_{s}^{t}H(u)dZ_{j}(u)\right|^{\eta}\right]\leq C(t-s)^{\eta/\alpha_{j}}\sup\limits _{u\in[s,t]}\E\left[|H(u)|^{\gamma}\right]^{\eta/\gamma}.
\]
\end{Lemma}

\subsection*{Acknowledgements}
The authors would like to thank Alexei Kulik for some interesting discussions
on ergodicity of Markov processes and, in particular,
on the meaning and sufficiency of condition (b) in Theorem \ref{th:3}.

\bibliographystyle{amsplain}
\bibliography{references}

\end{document}